\documentclass[11pt]{amsart}

\usepackage{fullpage, graphicx}
\usepackage{amssymb}

\usepackage[colorlinks=true]{hyperref}

\usepackage{manfnt} 
\usepackage{tikz}
\usepackage{dsfont}
\usepackage{color}
\usepackage{pstricks-add}
\usepackage{pst-node}
\usepackage{pst-tree}
\usepackage{float}

\newcommand{\lbb}{{[\mspace{-3 mu} [}}
\newcommand{\rbb}{{] \mspace{-3 mu} ]}}
\newcommand{\ZZ}{{\mathbb Z}}

\newtheorem{thm}{Theorem}[section]
\newtheorem{prop}[thm]{Proposition}

\newtheorem{lemma}[thm]{Lemma}

\newtheorem{definition}[thm]{Definition}
\theoremstyle{definition}
\newtheorem{example}[thm]{Example}

\title{A chord diagram expansion coming from some Dyson-Schwinger equations}
\author{Nicolas Marie and Karen Yeats}

\thanks{This research was supported by an NSERC discovery grant.  The authors would like to thank Dirk Kreimer and Humboldt University for hospitality and support.}

\begin{document}

\begin{abstract}
  We give an expression for the solution to propagator-type Dyson-Schwinger equations with one primitive at 1 loop as an expansion over rooted connected chord diagrams.  Along the way we give a refinement of a classical recurrence of rooted connected chord diagrams, and a representation in terms of binary trees.
\end{abstract}

\maketitle

\section{Introduction}

Dyson-Schwinger equations are integral equations in quantum field theory which have a recursive structure that mirrors the recursive decomposition of Feynman diagrams into subdiagrams.  As such, they have a strong combinatorial flavour and there is much which can be said about them using combinatorial tools.  They are also very difficult to solve in general, and even partial information can provide physically valuable results.  

This paper looks at the special case where the underlying decomposition of diagrams has the same shape as the standard recurrence for rooted trees.  This occurs in Dyson-Schwinger equations for propagators built from one diagram with one loop inserted in one place.  For an example and more details see Section \ref{DSE section}.  

The analytic contribution of this one diagram (which must be a primitive graph in the renormalization Hopf algebra, as so will be referred to as \textbf{the primitive} of the Dyson-Schwinger equation) can be given in the form of an expansion
\[
  F(\rho) = \frac{f_0}{\rho} + f_1 + f_2\rho+ \cdots
\]
We will view the $f_i$ as known, given to us by physics.
The Dyson-Schwinger equation we are interested in is
\begin{equation}\label{DSEintro}
 G(x,L) = 1- x G\left(x,\frac{d}{d(-\rho)}\right)^{-1} (e^{-L\rho}-1)F(\rho) \big|_{\rho=0}
\end{equation}
and the problem we consider is to give an explicit expression for $G(x,L)$ in terms of the $f_j$. 

Let $\mathcal{RCCD}$ be the set of rooted connected chord diagrams, or equivalently connected arc diagrams of matchings.  These are classical and purely combinatorial objects.  Let the size of a rooted be the number of chords.  For details and definitions on chord diagrams see Section \ref{RCCD section}. 
The main result of this paper is that
\[
G(x,L) = 1 - \sum_{i\geq 1}\frac{(-L)^i}{i!}\sum_{\substack{C\in \mathcal{RCCD} \\ b(C) \geq i}}x^{|C|}f_{C}
f_{b(C)-i} 
\]
solves \eqref{DSEintro}, where $f_{C}$ is a monomial in the $f_j$ given by the chord diagram, and $b(C)$ is a parameter of $C$ which can be read off the intersection graph of $C$.  The proof of this result is a mix of explicit combinatorial constructions and recurrences.
The result gives us $G(x,L)$, which is a physically meaningful quantity, as a sort of multivariate generating function of chord diagrams.   Some initial observations on the consequences of this are given in Section~\ref{conclusion section}.

For the reader who wishes to jump directly to the combinatorial constructions and the proof of the main result, sections~\ref{RCCD section} and \ref{rec section} can be read independently of the motivation from Section~\ref{DSE section}.

\medskip

The paper is organized as follows.  Section \ref{DSE section} briefly runs through the physical set up which gives rise to problem considered in the paper.  It is written with the mathematical reader in mind and contains references to more comprehensive sources.  Section \ref{RCCD section} sets out the definitions and constructions we need on rooted chord diagrams.  Section \ref{rec section} begins by proving two recurrences, one of which refines a classical recurrence of Stein, and the other of which is a consequence of our tree representation for rooted chord diagrams.  Together these two recurrences let us conclude Section \ref{rec section} with a proof of our main result.  Section \ref{conclusion section} gives an elementary account of consequences of the result. These consequences can be divided in an analytic and a combinatorial part, each of which will be explored further in a subsequent work.  The paper concludes with Appendix \ref{obj appendix} which gives a table of examples of rooted chord diagrams and their associated trees.

\section{Dyson-Schwinger equations}\label{DSE section}

Let $\mathcal{F}$ be the set of forests of rooted trees.  Such forests have a size given by the number of vertices.  Let $\mathcal{H}$ be the graded vector space $\text{span}(\mathcal{F})$ over a base field of characteristic $0$.  Defining the product of two forests to be their disjoint union and extending linearly makes $\mathcal{H}$ into an algebra: the polynomial algebra generated by rooted trees.  The identity element is the empty tree and will be denoted $1$.  In fact, $\mathcal{H}$ has a nice combinatorially defined coproduct and is the Connes-Kreimer Hopf algebra of rooted trees \cite{ck0}.  This Hopf algebra structure underlies the quantum field theoretic context of the problem currently at hand, however we don't need it for the purposes of this paper and so it will be left to the references.

\begin{definition}
  Let $B_+ : \mathcal{H} \rightarrow \mathcal{H}$ be the operation which takes a forest $T_1 T_2 \cdots T_k$ and returns the rooted tree where the subtrees rooted at children of the root are $T_1, T_2, \ldots , T_k$, and extended linearly to all of $\mathcal{H}$.
\end{definition}

Using this notation we can capture the recursive decomposition of a tree into the subtrees of the root by the following equation
\begin{equation}\label{comb DSE}
  X(x) = 1 - xB_+\left(\frac{1}{X(x)}\right)
\end{equation}
Note that since $X(x)$ begins with $1$, $1/X(x)$ is well defined simply by expanding the geometric series.
The solution to this equation in $\mathcal{H}[[x]]$, which can be constructed recursively and checked inductively, is
\[
  X(x) = 1 - \sum_{T}x^{|T|}p(T)T
\]
where the sum runs over all rooted trees and $p(T)$ is the number of distinct plane representations of a tree $T$, or equivalently the number of nonisomorphic orderings of the children of each vertex.  Other such equations yield other classes of trees \cite{bergk}.  Mapping each tree to $1$ we would get the usual generating function; we will be interested in more subtle maps from trees and graphs to numbers.

The same basic structure, though in general considerably more intricate, is seen in the nesting of subdivergent Feynman diagrams in larger Feynman diagrams.  This observation is the beginning of the algebraic or combinatorial approach to Dyson-Schwinger equations as found in papers such as \cite{anatomy, kythesis, etude, vS}.  \textbf{Combinatorial Dyson-Schwinger equations} are equations such as \eqref{comb DSE} and its generalizations which describe this recursive structure for a given class of Feynman graphs.

In some simple instances the combinatorial Dyson-Schwinger equation will be exactly in the same form as \eqref{comb DSE}.  For example \cite{bkerfc} considers two such cases.  One of these cases is the case where 
\[
\includegraphics{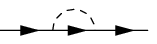}
\] 
is inserted into itself in all possible ways, yielding graphs such as
\[
\includegraphics{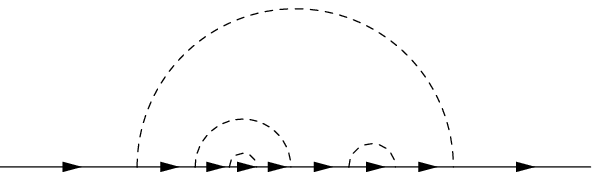}
\]
which has a tree structure to its insertions.
This form will appear whenever we are interested in iterating a propagator graph into itself along one internal edge.  More general Dyson-Schwinger equations are more complicated as different sorts of graphs can be inserted in different places, and subgraphs can overlap.

\textbf{Feynman rules} give a map from Feynman graphs into (formal) integrals.  One can then regularize the integrals, renormalize them, and sum over appropriate graphs to obtain physically meaningful values.  A reference in more or less the language used here is \cite{e-fksurvey}.

Applying Feynman rules to the combinatorial Dyson-Schwinger equations gives \textbf{analytic Dyson-Schwinger equations} which are integral equations for the Green functions of the system in question.  They have the same basic recursive structure that the combinatorial Dyson-Schwinger equations did.  Continuing the above example \cite{bkerfc} the combinatorial Dyson-Schwinger equation
\[
  X(x) = \mathbb{I} - x B_+
\left(\frac{1}{X(x)}\right)
\]
interpreted as inserting the graph
\[
\includegraphics{Yukawaeg}
\]
into itself iteratively in all possible ways, yields the analytic Dyson-Schwinger equation
\[
G(x, L) = 1 - \frac{x}{q^2}\int d^4 k \frac{k \cdot q}{k^2 G(x,
  \log k^2)(k+q)^2} - \cdots \bigg|_{q^2 = \mu^2}
\]
where $L = \log(q^2/\mu^2)$.  See \cite{kythesis} (also available as \cite{Ymem}) Example 3.5 for further details.

The analytic Dyson-Schwinger equations are still not in the form we need. Example 3.7 of \cite{kythesis, Ymem} begins with the above example, proceeds to expand $G(x,L)$ in $L$, convert logarithms to powers using $\frac{d^ky^{\rho}}{d\rho^k}  |_{\rho=0} = \log^k(y)$, swap the order of the operators, and thus obtains
\[
 G(x,L) = 1- x G\left(x,\frac{d}{d(-\rho)}\right)^{-1} (e^{-L\rho}-1)F(\rho) \big|_{\rho=0}
\]
where $F(\rho)$ is the Feynman integral of the primitive with the propagator we are inserting on regularized, and the integral evaluated at $q^2=1$.

Rather than set up appropriate hypotheses on the original analytic Dyson-Schwinger equations so as to guarantee that these transformations are always possible (which they will be in physically reasonable circumstances), we will now follow \cite{kythesis} by \emph{defining} the analytic Dyson-Schwinger equation associated to \eqref{comb DSE} to be

\begin{equation}\label{DSE}
  G(x,L) = 1- x G\left(x,\frac{d}{d(-\rho)}\right)^{-1} (e^{-L\rho}-1)F(\rho) \big|_{\rho=0}
\end{equation}
where
\[
  F(\rho) = \frac{f_{0}}{\rho} + f_1 + f_2\rho + f_3\rho^2 + \cdots
\]
which is the integral of the primitive graph regularized by raising the insertion propagator to the power $1+\rho$, evaluated with external momentum equal to $1$, and expanded in $\rho$.  $F(\rho)$ is viewed as given.

Write
\[
G(x,L) = 1-\sum_{k \geq 1}\gamma_k(x)L^k
\]
and view the $\gamma_k(x)$ as unknown series in $x$.  Expanding it out, one can see that \eqref{DSE} determines the $\gamma_k(x)$ in terms of the $f_i$, but only in a quite unwieldy way.  

The goal of \cite{kythesis,Ymem} was to convert considerably more general Dyson-Schwinger equations into a more workable form at the cost of introducing a new unknown function $P$.  This new form proved quite useful, in \cite{vBKUY,vBKUY2}  one of us along with Dirk Kreimer, Guillaume van Baalen, and David Uminsky use the new form to investigate QED and QCD showing in the former case that it is possible to avoid the Landau pole.  Marc Bellon and Fidel Schaposnik \cite{BellonII, BellonDE, BellonDSE, BShigher} have investigated approximation schemes based on this method.  

The present paper is more modest in the sense that only the Dyson-Schwinger equation \eqref{DSE} is considered.  However in another sense it is better; the result gives an explicit solution to this Dyson-Schwinger equation as an expansion indexed by rooted chord diagrams; no mystery functions or recurrences.  In the last section we give some consequences of this expansion and we will pursue the physical consequences more fully in future work with Dirk Kreimer.

\section{Rooted chord diagrams}\label{RCCD section} 

In this section we will define the combinatorial objects and constructions we need.
 
At the heart of the proof that the formal solution to the Dyson--Schwinger equation \eqref{DSEintro} is given by
\begin{equation*}
G(x,L)\,=\,1 - \sum_{i \in \mathbb{N^*}} \frac{(-L)^i}{i!} \, \sum_{\substack{ C \in \mathcal{RCCD} \\ b(C) \geq i}} f_{C}\,f_{b(C)-i}\,x^{|C|}
\end{equation*}
is the combinatorics of rooted connected chord diagrams. In particular, the monomials in the $f_{j}$ appearing in this expansion are indexed by the sequences of gaps between distinguished chords called the terminal chords of the diagrams.\\

We will describe these structures as well as a representation of the chord diagrams by planar binary trees whose decomposition is the main ingredient of the recurrences in section $4$.

\subsection{Chord diagrams}

\begin{definition}
A \textbf{chord diagram} of order n is the data of $2n$ points $(p_{1},\cdots,p_{2n})$ arranged on a circle, together with $n$ distinct pairs of distinct points $\left\{ \{p_{i_{1}},p_{i_{2}}\}, \cdots, \{p_{i_{n}},p_{i_{2n}}\} \right\}$ where $\{p_{i_{k}},p_{i_{2k}} \}$ is represented by a chord joining the points $p_{i_{k}}$ and $p_{i_{2k}}$ on the circle. It is \textbf{rooted} when we distinguish one of those points on the circle. It is \textbf{connected} when after erasing the circle we are left with a connected graph.
\end{definition}

In the following $\mathcal{RCCD}$ denote the set of rooted connected chord diagrams. We say that a diagram with $n$ chords has degree $n$ and we denote by $\mathcal{RCCD}(n)$ the family of these diagrams.
Unless stated otherwise our chord diagrams are oriented counterclockwise, meaning that the points on the outer circle are numbered $p_{1},p_{2},...,p_{2n}$ with $p_{1}$ corresponding to the root and continuing counterclockwise from this point.

\begin{figure}[H]

 \begin{center}
 \hspace{3cm}
\begin{pspicture}(0,0)(12.8,3.2)
\pscircle(1,1.900){1.25}
\psarc[linecolor=black]{*-*}(1.8,0.65){0.9}{68}{178}
\psarc[linecolor=black]{*-*}(3,3.3){2.5}{186}{244}
\psarc[linecolor=black]{*-*}(0.1,4){1.5}{263}{323}
\psarc[linecolor=black]{*-*}(0.1,4){1.9}{260}{326}
\pscircle[linecolor=red](0.9,0.68){0.12}
\rput(1,0.2){\tiny(A)}
\rput(0.8,0.48){\tiny$1$}
\rput(2,0.868){\tiny$2$}
\rput(1.8,3.11){\tiny$3$}
\rput(1.4,3.29){\tiny$4$}

\pscircle(8,1.900){1.25}
\psarc[linecolor=black]{*-*}(8.8,0.65){0.9}{68}{178}
\psarc[linecolor=black]{*-*}(7.1,4){1.5}{263}{323}
\psarc[linecolor=black]{*-*}(6.8,-0.2){2}{30}{91}
\pscircle[linecolor=red](7.9,0.68){0.12}
\rput(8,0.2){\tiny(B)}
\rput(7.8,0.48){\tiny$1$}
\rput(8.6,0.62){\tiny$2$}
\rput(8.42,3.28){\tiny$3$}


\end{pspicture}
\end{center}

\caption{ \small $(A)$ is a rooted connected chord diagram of degree $4$ while $(B)$ has degree $3$ but is not connected. Here the root is the circled vertex, the root chord is numbered $1$ and the rest is labelled in the counterclockwise order.}
\end{figure}
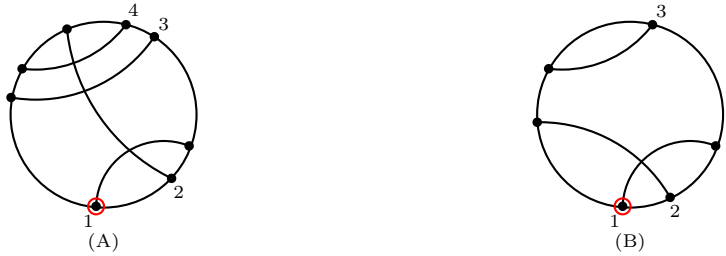 \hspace{0.5cm}

\noindent Observe that in general a chord diagram is distinct from its mirror image as can be seen in the following example:

\begin{figure}[H]

 \begin{center}
 \hspace{3cm}
\begin{pspicture}(0,0)(12.8,3.2)
\pscircle(1,1.900){1.25}
\psarc[linecolor=black]{*-*}(1.8,0.65){0.9}{68}{178}
\psarc[linecolor=black]{*-*}(3,3.3){2.5}{186}{244}
\psarc[linecolor=black]{*-*}(0.1,4){1.5}{263}{323}
\psarc[linecolor=black]{*-*}(0.1,4){1.9}{260}{326}
\pscircle[linecolor=red](0.9,0.68){0.12}
\rput(0.8,0.48){\tiny$1$}
\rput(2,0.868){\tiny$2$}
\rput(1.8,3.11){\tiny$3$}
\rput(1.4,3.29){\tiny$4$}

\rput(15,0){
\psscalebox{-1 1}{
\pscircle(8,1.900){1.25}
\psarc[linecolor=black]{*-*}(8.8,0.65){0.9}{68}{178}
\psarc[linecolor=black]{*-*}(7.1,4){1.5}{263}{323}
\psarc[linecolor=black]{*-*}(10,3.3){2.5}{186}{244}
\psarc[linecolor=black]{*-*}(7.1,4){1.9}{260}{326}
\pscircle[linecolor=red](7.9,0.68){0.12}
}
}
\psline[linestyle=dashed,linewidth=0.1cm]{C-C}(4.1,0.25)(4.1,3.5)
\rput(10,1.2){.}
\rput(7.2,0.38){\tiny$1$}
\rput(8.5,2.1){\tiny$2$}
\rput(8.3,2.6){\tiny$3$}
\rput(7.6,3.2){\tiny$4$}


\end{pspicture}
\end{center}
\end{figure} \hspace{0.5cm}

\begin{definition}
Let $X$ be a rooted connected chord diagram of degree $n$. We denote by $\mathcal{I}(X)$ the labeled directed graph whose set of vertices $\{1,2,...,i,...,n\}$ corresponds to the set of chords of $X$ where $i$ stands for the $i^{\text{th}}$ chords in the counterclockwise order and there is a directed edge from $i$ to $j$ if the $i^{\text{th}}$ chord intersects the $j^{\text{th}}$ chord with $i<j$. The graph $\mathcal{I}(X)$ is called the \textbf{directed intersection diagram} of $X$.
\end{definition}

Here is an example of a rooted connected chord diagram and its directed intersection graph:

\begin{figure}[H]

 \begin{center}
 \hspace{3cm}
\begin{pspicture}(0,0)(12.8,3)
\pscircle(1,1.900){1.25}
\psarc[linecolor=black]{*-*}(1.8,0.65){0.9}{68}{178}
\psarc[linecolor=black]{*-*}(3,3.3){2.5}{186}{244}
\psarc[linecolor=black]{*-*}(0.1,4){1.5}{263}{323}
\psarc[linecolor=black]{*-*}(0.1,4){1.9}{260}{326}
\pscircle[linecolor=red](0.9,0.68){0.12}
\rput(0.8,0.48){\tiny$1$}
\rput(2,0.868){\tiny$2$}
\rput(1.8,3.11){\tiny$3$}
\rput(1.4,3.29){\tiny$4$}

\rput(5,1.9){1}
\psline{->}(5.1,1.9)(6.4,1.9)
\rput(6.5,1.9){2}
\psline{->}(6.6,1.9)(7.9,1.9)
\rput(8,1.9){3}
\psbezier[linestyle=solid]{->}(6.6,2.1)(8,2.5)(8,2.5)(9.4,2.1)
\rput(9.5,1.9){4}

\rput(11.5,1.9){.}

\end{pspicture}
\end{center}
\end{figure} \vspace{-0.5cm}

\begin{definition}
Let $X$ be a rooted connected chord diagram and $\mathcal{I}(X)$ its directed intersection diagram. A chord $i$ is said to be \textbf{terminal} if the vertex $i$ of $\mathcal{I}(X)$ has no outgoing edges.
\end{definition}

Hence a terminal chord does not intersect any chord with a larger index. As one can see in the examples presented below, the linear order on the vertices of the directed intersection graphs makes it easy to observe the gaps between terminal chords. \\

But unfortunately things are not straightforward and at this point we need to relabel the chords of our diagrams in a new order called the intersection order.\\

This order is defined recursively directly on the chord diagrams or on their intersection diagrams. We choose this second option and we express this new order as a permutation of the counterclockwise order.

\begin{definition}
Let $X$ be a rooted connected chord diagram of degree $n$ with its sequence of chords $(1,2,...,n)$ in the counterclockwise order and $\mathcal{I}(X)$ its intersection diagram. Apply the following recursive procedure:\\
$1)$ consider the graph $\mathcal{I}(X)$ and delete the edges going out of its smallest vertex, the vertex $1$;\\
$2)$ obtain $k$ connected components $\mathcal{I}_{1}(X)=\{ 1 \},\mathcal{I}_{2}(X),...,\mathcal{I}_{k}(X)$ where the smallest vertex of $\mathcal{I}_{p}(X)$ is larger than the smallest vertex of $\mathcal{I}_{q}(X)$ when $q<p$;\\
$3)$ then each connected component $\mathcal{I}_{p}(X)$ is associated to its sequence of vertices $(x_{1,p},x_{2,p},...)$ in counterclockwise order. This defines a permutation $(1,2,...,n) \mapsto (1,x_{1,1},x_{2,1},...,x_{1,2},x_{2,2},...,x_{1,k},x_{2,k},...)$;
$4)$ apply this procedure recursively to each $\mathcal{I}_{p}(X),(x_{1,p},x_{2,p},...)$ until we are left with $n$ singleton;\\
This defines a permutation $\sigma_{X}: (1,2,...,n) \mapsto (\sigma_{1},\sigma_{2},...,\sigma_{n})$ that we call the \textbf{intersection order} of $X$.
\end{definition}
\vspace{0.5cm}

This procedure is easily understood on an example:\\ \vspace{1cm}
\begin{figure}[H]

 \begin{center}
 \hspace{3cm}
\begin{pspicture}(0,0)(12.8,3.2)

\pscircle(-2,1.900){1.25}
\psrotate(-2,1.9){45}{\psline{*-*}(-2,0.65)(-2,3.15)}
\psrotate(-2,1.9){-45}{\psline{*-*}(-2,0.65)(-2,3.15)}
\psrotate(-2,1.9){55}{\psarc[linecolor=black]{*-*}(-2,3.75){1}{230}{310}}
\psrotate(-2,1.9){-55}{\psarc[linecolor=black]{*-*}(-2,3.75){1}{230}{310}}

\pscircle[linecolor=red](-2.89,1.01){0.12}
\rput(-2.89,0.75){\tiny$1$}
\rput(-1.0,0.8){\tiny$2$}
\rput(-0.55,2){\tiny$3$}
\rput(-2.57,3.25){\tiny$4$}
\rput(-3.5,0.55){\tiny$(X)$}

\rput(-2,-0.25){$\sigma_{X}=(1243)$}

\rput(2,3.5){1}
\psline{->}(2.1,3.5)(3.4,3.5)
\psbezier[linestyle=solid]{->}(2.1,3.7)(3.5,4)(3.5,4)(4.9,3.7)
\rput(3.5,3.5){2}
\psbezier[linestyle=solid]{->}(3.6,3.3)(5,3.1)(5,3.1)(6.4,3.3)
\rput(5,3.5){3}
\rput(6.5,3.5){4}

\rput(2,1.5){1}
\rput(3.5,1.5){2}
\psline{->}(3.6,1.5)(4.9,1.5)
\rput(5,1.5){4}
\rput(6.5,1.5){3}

\rput(2,-0.5){1}
\rput(3.5,-0.5){2}
\rput(5,-0.5){3}
\rput(6.5,-0.5){4}

\rput(10,4.2){\textbf{Step 1}}
\rput(10,3.5){$(1,2,3,4)$}
\rput(10,2.2){\textbf{Step 2}}
\rput(10,1.5){$(1,2,4,3)$}
\rput(10,0.2){\textbf{Step 3}}
\rput(10,-0.5){$(1,2,4,3)$}

\rput(4.25,2.8){\small$\mathcal{I}(X)$}
\rput(2,0.8){\small$\mathcal{I}_{1}(X)$}
\rput(4.25,0.8){\small$\mathcal{I}_{2}(X)$}
\rput(6.5,0.8){\small$\mathcal{I}_{3}(X)$}
\rput(2,-1.1){\small$\mathcal{I}_{1}(X)$}
\rput(3.5,-1.1){\small$\mathcal{I}_{2,1}(X)$}
\rput(5,-1.1){\small$\mathcal{I}_{2,2}(X)$}
\rput(6.5,-1.1){\small$\mathcal{I}_{3}(X)$}

\end{pspicture}
\end{center}

\end{figure} \vspace{1cm}

 where each step corresponds to rearranging the elements of each block with respect to their smallest label as prescribed in the definition.\\

\noindent So finally we obtain the chord diagram and its corresponding intersection diagram now labeled in the intersection order given by $\sigma_{X}=(1243)$: 

\begin{figure}[H]

 \begin{center}
 \hspace{3.5cm}
\begin{pspicture}(0,0)(12.8,3)

\pscircle(-0,1.900){1.25}
\psrotate(-0,1.9){45}{\psline{*-*}(-0,0.65)(-0,3.15)}
\psrotate(-0,1.9){-45}{\psline{*-*}(-0,0.65)(-0,3.15)}
\psrotate(-0,1.9){55}{\psarc[linecolor=black]{*-*}(-0,3.75){1}{230}{310}}
\psrotate(-0,1.9){-55}{\psarc[linecolor=black]{*-*}(-0,3.75){1}{230}{310}}

\pscircle[linecolor=red](-0.89,1.01){0.12}
\rput(-0.89,0.75){\tiny$1$}
\rput(1.0,0.8){\tiny$2$}
\rput(1.55,2){\tiny$4$}
\rput(-0.57,3.25){\tiny$3$}
\rput(-1.5,0.55){\tiny$(X)$}


\rput(6,2){1}
\psline{->}(6.1,2)(7.4,2)
\psbezier[linestyle=solid]{->}(6.1,2.2)(8,2.5)(8,2.5)(10.4,2.2)
\rput(7.5,2){2}
\psline{->}(7.6,2)(8.9,2)
\rput(9,2){3}
\rput(10.5,2){4}

\rput(8.25,1){\small$\mathcal{I}_{\sigma}(X)$}

\rput(11.25,2){.}

\end{pspicture}
\end{center}

\end{figure} \vspace{0cm}

If $\mathcal{I}(X)$ is a directed intersection graph we denote by $\mathcal{I}_{\sigma}(X)$ the graph obtained by relabelling the vertices with the permutation $\sigma_{X}$. This operation is an automorphism of the graph $\mathcal{I}(X)$ so in if $i$ was a terminal (respectively initial) chord of $X$ in the counterclockwise order, $\sigma_{i}=\sigma_{X}(i)$ is a terminal (respectively initial) chord of $X$ in the intersection order.

For chord diagrams of small degree the intersection order and the counterclockwise order coincide most of the time. It is only at higher degrees that we start to see the differences between these orders as one can observe on some families of examples given below.\\

\begin{definition}
Let $X$ be a rooted connected chord diagram with intersection order $\sigma_{X}=(\sigma_{1},...,\sigma_{n})$. The sequence of terminal chords of $X$ in the intersection order is denoted by $Ter_{\sigma}(X)=(\sigma_{i_{1}},...,\sigma_{i_{k}})$ with $\sigma_{i_{1}}<\sigma_{i_{2}}<...<\sigma_{i_{k}}$. We associate to $Ter_{\sigma}(X)$ the sequence of consecutive gaps between terminal chords in the intersection order:
\begin{equation*}
\delta(X)=(\sigma_{i_{2}}-\sigma_{i_{1}},\sigma_{i_{3}}-\sigma_{i_{2}},...,\sigma_{i_{k}}-\sigma_{i_{k-1}})\,.
\end{equation*}
\end{definition}

We denote by $b(X)$ the first element of $Ter_{\sigma}(X)$ i.e. the smallest terminal chord in the intersection order.\\

It is easier to handle sequences of gaps $\delta(X)$ with constant lengths over the chord diagrams with constant degree. So if $X$ has degree $n$ and $\delta(X)=(\delta_{1},...,\delta_{k})$ we introduce:
\begin{equation*}
\bar{\delta}(X)=\underbrace{(0,...,0,}_{n-k-1 \text{times}}\delta_{1},...,\delta_{k})\,,
\end{equation*}

so that $\bar{\delta}(X)$ has length $n-1$ if $X$ has degree $n$.\\

These are the objects that index the monomials in $f_{j}$ appearing in the $\gamma_{k}$ expansions and we define:
\begin{equation*}
f_{X}=f_{0}^{n-k-1}\,f_{\delta_{1}}\,f_{\delta_{2}}\,...\,f_{\delta_{k}}\,.
\end{equation*}

Now that we understand the correspondence between the monomials in the chord diagram expansion and the gaps separating terminal chords we need to describe a decomposition of the chord diagrams.
Any rooted connected chord diagram of degree $n$ defines $2n$ intervals on the outer circle that we label $0,1,2,....,2n-1$ starting with $0$ for the interval preceding the root and progressing counterclockwise:

\begin{figure}[H]

 \begin{center}
 \hspace{3.5cm}
\begin{pspicture}(0,0)(10,3)

\psscalebox{1.11}{
\pscircle(-0.5,1.9){1.25}
\psline{*-*}(-0.5,0.65)(-0.5,3.15)
\psarc[linecolor=black]{*-*}(-0.5,3.15){1.3}{212}{328}
\pscircle[linecolor=red](-0.5,0.65){0.12}
\rput(-1.1,1.1){\tiny$0$}
\rput(0.15,1.1){\tiny$1$}
\rput(0.15,2.7){\tiny$2$}
\rput(-1.1,2.7){\tiny$3$}
}

\psscalebox{1}{
\pscircle[fillstyle=solid,fillcolor=gray](5.1,1.9){0.5}
\psline{*-}(5.1,0.4)(5.1,1.4)
\psrotate(5.1,1.9){30}{\psline{*-}(5.1,0.4)(5.1,1.4)}
\psrotate(5.1,1.9){-30}{\psline{*-}(5.1,0.4)(5.1,1.4)}
\psrotate(5.1,1.9){60}{\psline{*-}(5.1,0.4)(5.1,1.4)}
\psrotate(5.1,1.9){180}{\psline{*-}(5.1,0.4)(5.1,1.4)}
\psrotate(5.1,1.9){210}{\psline{*-}(5.1,0.4)(5.1,1.4)}
\psarc[linecolor=black,linestyle=dashed]{-}(5.1,1.9){1.5}{0}{70}
\psarc[linecolor=black,linestyle=solid]{-}(5.1,1.9){1.5}{70}{140}
\psarc[linecolor=black,linestyle=dashed]{-}(5.1,1.9){1.5}{145}{220}
\psarc[linecolor=black,linestyle=solid]{-}(5.1,1.9){1.5}{225}{360}
\pscircle[linecolor=red](5.1,0.4){0.12}
\rput(4.8,0.6){\tiny$0$}
\rput(5.4,0.6){\tiny$1$}
\rput(6,1){\tiny$2$}
\rput(4.8,3.2){\tiny$k$}
\rput(8,1.9){.}
}

\end{pspicture}
\end{center}
\end{figure} \vspace{-0.9cm}\hspace{0.55cm}

\subsection{Decomposition and trees}

The rest of the constructions are based on the following insertion operation defined on rooted connected chord diagrams.

\begin{definition}
Let $m,n\in \mathbb{N}^*$ and $0<i\leq 2n-1$. Define the operation $\stackrel{\curvearrowright}{(0,i)} : \mathcal{RCCD}(m) \times \mathcal{RCCD}(n) \longrightarrow \mathcal{RCCD}(m+n)$ for all $X \in \mathcal{RCCD}(m) $ and  $Y \in \mathcal{RCCD}(n) $ by

\begin{figure}[H]

 \begin{center}
 \hspace{1cm}
\begin{pspicture}(0,0)(12.8,3.2)
\psarc[linecolor=black,linestyle=solid]{-}(1,1.9){1.5}{-15}{15}
\psarc[linecolor=black,linestyle=dashed]{-}(1,1.9){1.5}{20}{70}
\psarc[linecolor=black,linestyle=solid]{-}(1,1.9){1.5}{75}{105}
\psarc[linecolor=black,linestyle=dashed]{-}(1,1.9){1.5}{110}{160}
\psarc[linecolor=black,linestyle=solid]{-}(1,1.9){1.5}{165}{195}
\psarc[linecolor=black,linestyle=dashed]{-}(1,1.9){1.5}{200}{240}
\psarc[linecolor=black,linestyle=solid]{-}(1,1.9){1.5}{245}{295}
\psarc[linecolor=black,linestyle=dashed]{-}(1,1.9){1.5}{300}{340}
\pscircle[fillstyle=solid,fillcolor=gray](1,1.9){0.5}
\psline[linecolor=red]{*-}(1,0.4)(1,1.4)
\psrotate(1,1.9){90}{\psline{*-}(1,0.4)(1,1.4)}
\psrotate(1,1.9){180}{\psline[linecolor=red]{*-}(1,0.4)(1,1.4)}
\psrotate(1,1.9){270}{\psline{*-}(1,0.4)(1,1.4)}
\pscircle[linecolor=red](1,0.4){0.12}
\rput(1,1.9){$X$}

\rput(3.25,2){$\stackrel{\curvearrowright}{(0,i)}$}

\psarc[linecolor=black,linestyle=dashed]{-}(5.5,1.9){1.5}{110}{220}
\psarc[linecolor=black,linestyle=dashed]{-}(5.5,1.9){1.5}{320}{345}
\psarc[linecolor=black,linestyle=solid]{-}(5.5,1.9){1.5}{-5}{100}
\psarc[linecolor=black,linestyle=solid]{-}(5.5,1.9){1.5}{230}{315}
\pscircle[fillstyle=solid,fillcolor=gray](5.5,1.9){0.5}
\psline{*-}(5.5,0.4)(5.5,1.4)
\psrotate(5.5,1.9){90}{\psline{*-}(5.5,0.4)(5.5,1.4)}
\psrotate(5.5,1.9){180}{\psline{*-}(5.5,0.4)(5.5,1.4)}
\psrotate(5.5,1.9){35}{\psline[linestyle=solid]{*-}(5.5,0.4)(5.5,1.4)}
\psrotate(5.5,1.9){325}{\psline[linestyle=solid]{*-}(5.5,0.4)(5.5,1.4)}
\pscircle[linecolor=red](5.5,0.4){0.12}
\rput(5.2,0.65){\tiny$0$}
\rput(5.82,0.65){\tiny$1$}
\rput(6.55,2.75){\small$i$}
\rput(5.5,1.9){$Y$}

\rput(7.75,1.9){$=$}


\psarc[linecolor=black,linestyle=dashed]{-}(10,1.9){1.5}{270}{350}
\psarc[linecolor=black,linestyle=dashed]{-}(10,1.9){1.5}{35}{65}
\psarc[linecolor=black,linestyle=dashed]{-}(10,1.9){1.5}{125}{210}

\psarc[linecolor=black,linestyle=solid]{-}(10,1.9){1.5}{-5}{30}
\psarc[linecolor=black,linestyle=solid]{-}(10,1.9){1.5}{70}{120}
\psarc[linecolor=black,linestyle=solid]{-}(10,1.9){1.5}{210}{270}

\pscircle[fillstyle=solid,fillcolor=gray](9.85,1.5){0.35}
\pscircle[fillstyle=solid,fillcolor=gray](10.5,2.35){0.35}
\psbezier[linestyle=solid,linecolor=red]{*-}(9.37,0.54)(9.6,1.5)(9.6,1.5)(10.25,2.1)
\psbezier[linestyle=solid,linecolor=red]{-*}(10.75,2.6)(10.95,2.7)(10.95,2.7)(11.05,2.95)
\psrotate(10,1.9){40}{\psline{*-}(8.5,1.9)(9.3,1.85)}
\psrotate(10,1.9){80}{\psline{*-}(8.5,1.9)(9.249,1.82)}
\psrotate(10,1.9){180}{\psline{*-}(8.5,1.9)(9.85,2.15)}
\psrotate(10,1.9){200}{\psline{*-}(8.5,1.9)(9.08,1.8)}
\psrotate(10,1.9){260}{\psline{*-}(8.5,1.9)(9.19,2.10)}
\psrotate(10,1.9){290}{\psline{*-}(8.5,1.9)(10.0,1.700)}
\pscircle[linecolor=red](9.37,0.54){0.12}
\rput(9.85,1.5){$Y$}
\rput(10.5,2.35){$X$}

\end{pspicture}
\end{center}
\end{figure} \vspace{-0.8cm}\hspace{-0.55cm}
that is the operation of insertion of $X$ in $Y$ by placing the root of $X$ in the interval $0$ of $Y$ and the rest of $X$ in the $i^{\text{th}}$ interval of $Y$. The root of $X \stackrel{\curvearrowright}{(0,i)}Y$ is the root of $X$. 
\end{definition}

Observe that any rooted connected chord diagram can be decomposed canonically using this insertion operation. A share in a chord diagram is formed by two arcs on the outer circle such that if one endpoint of a chord is in the share the other endpoint must also be in the share, e.g.

\begin{figure}[H]

 \begin{center}
 \hspace{3cm}
\begin{pspicture}(0,0)(12.8,3.8)

\psscalebox{1.2}{
\psccurve[showpoints=false,linestyle=dashed,fillcolor=lightgray,fillstyle=solid](5.38,0.6)(5.2,0.81)(5.8,1.9)(6.2,2.7)(6.1,3.4)(6.7,3.2)(7.5,1.8)(6.7,2.1)(6.1,1.8)(5.65,0.65)
}

\psscalebox{1.2}{
\pscircle(1,1.900){1.25}
\psrotate(1,1.9){-75}{\psarc[linecolor=black]{*-*}(3,3.3){2.5}{186}{244}}
\psrotate(1,1.9){-70}{\psarc[linecolor=black]{*-*}(0.1,4){1.5}{263}{323}}
\psrotate(1,1.9){175}{\psarc[linecolor=black]{*-*}(3,3.3){2.5}{186}{244}}
\psarc[linecolor=red]{|-|}(1,1.9){1.35}{210}{265}
\psarc[linecolor=red]{|-|}(1,1.9){1.35}{-5}{90}
\psarc[linecolor=black]{*-*}(-0.9,3.4){1.5}{296}{347}
\pscircle[linecolor=red](0.41,0.81){0.12}
}
\rput(0.1,0.8){$a_-$}
\rput(2.9,3.2){$a_+$}
\rput(4.3,2){\huge$\longleftrightarrow$}

\psscalebox{1.2}{
\pscircle(6,1.900){1.25}
\psrotate(6,1.9){-75}{\psarc[linecolor=black]{*-*}(8,3.3){2.5}{186}{244}}
\psrotate(6,1.9){-70}{\psarc[linecolor=black]{*-*}(5.1,4){1.5}{263}{323}}
\psrotate(6,1.9){175}{\psarc[linecolor=black]{*-*}(8,3.3){2.5}{186}{244}}
\psarc[linecolor=black]{*-*}(4.1,3.4){1.5}{296}{347}
\pscircle[linecolor=red](5.41,0.81){0.12}
}




\end{pspicture}
\end{center}

\end{figure} \vspace{-0.5cm}

\noindent the grey area corresponds to the share of the arcs $(a_-,a_+)$.

\begin{definition}
Let $X$ be a rooted connected chord diagram and $\dot{X}$ the share of $X$ formed on one side by the root and on the other side by what is left of the diagram after deleting consecutively the root-chord of $X$ then the first connected component, for the induced counterclockwise order, of $X$ minus the root-chord. Then the \textbf{root-share decomposition} of $X$ is
\begin{equation*}
X=\dot{X} \stackrel{\curvearrowright}{(0,i)} (X\setminus \dot{X})
\end{equation*}
where $X\setminus \dot{X}$ is the diagram obtained by removing the chords of $\dot{X}$ from $X$ and taking the root to be the second chord in the counterclockwise order of $X$.
\end{definition}

\noindent Here is an example of root-share decomposition:

\begin{figure}[H]

 \begin{center}
 \hspace{3cm}
\begin{pspicture}(0,0)(12.8,3.2)

\psccurve[showpoints=false,linestyle=dashed,fillcolor=lightgray,fillstyle=solid](0.38,0.6)(0.2,0.81)(0.8,1.9)(1.2,2.7)(1.1,3.4)(1.7,3.2)(2.5,1.8)(1.7,2.1)(1.1,1.8)(0.6,0.7)

\pscircle(1,1.900){1.25}
\psrotate(1,1.9){-75}{\psarc[linecolor=black]{*-*}(3,3.3){2.5}{186}{244}}
\psrotate(1,1.9){-70}{\psarc[linecolor=black]{*-*}(0.1,4){1.5}{263}{323}}
\psrotate(1,1.9){175}{\psarc[linecolor=black]{*-*}(3,3.3){2.5}{186}{244}}
\psarc[linecolor=black]{*-*}(-0.9,3.4){1.5}{296}{347}
\pscircle[linecolor=red](0.41,0.81){0.12}
\rput(3.2,1.9){$=$}

\pscircle(5,1.900){1.25}
\psrotate(5,1.9){-75}{\psarc[linecolor=black]{*-*}(7,3.3){2.5}{186}{244}}
\psrotate(5,1.9){-70}{\psarc[linecolor=black]{*-*}(4.1,4){1.5}{263}{323}}
\pscircle[linecolor=red](4.41,0.81){0.12}

\rput(7,1.9){\small$\stackrel{\curvearrowright}{(0,1)}$}

\pscircle(9,1.900){1.25}
\psrotate(9,1.9){175}{\psarc[linecolor=black]{*-*}(11,3.3){2.5}{186}{244}}
\psarc[linecolor=black]{*-*}(7.1,3.4){1.5}{296}{347}
\pscircle[linecolor=red](9.39,0.71){0.12}

\rput(11,1.9){.}

\end{pspicture}
\end{center}

\end{figure} \vspace{-0.5cm}

Then we can use the root-share decomposition to construct planar binary trees from those chord diagrams. Let $\mathcal{PBT}$ denote the set of rooted planar binary trees counted by their number of leaves. We label the edges of those trees from $1$ to $2n-1$  by a preorder traversal starting with $1$ for the root edge. Here are some examples:

\begin{figure}[H]

 \begin{center}
 \vspace{0.5cm}
 \hspace{1cm}
\begin{pspicture}(0,0)(12.8,3.2)

\psscalebox{1.4}{
\rput(2,1.8){
\pstree[treesep=0.5cm,levelsep=0.5cm]{\Tc*{0pt}}
{\pstree[treesep=1.2cm,levelsep=0.5cm]{\Tc*{1.2pt}}{\Tc*{1.2pt} \pstree[treesep=1.2cm,levelsep=0.5cm]{\Tc*{0pt}}{\Tc*{1.2pt}  \Tc*{1.2pt} } } }
}
\pscircle[linecolor=black](1.68,2.07){0.1}
\rput(1.55,2.4){\tiny$1$}
\rput(1.15,1.9){\tiny$2$}
\rput(2.18,1.9){\tiny$3$}
\rput(1.85,1.4){\tiny$4$}
\rput(2.8,1.4){\tiny$5$}
}

\psscalebox{1.4}{
\rput(6,1.8){
\pstree[treesep=0.5cm,levelsep=0.5cm]{\Tc*{0pt}}
{\pstree[treesep=1.2cm,levelsep=0.5cm]{\Tc*{1.2pt}}{\Tc*{1.2pt} \pstree[treesep=1.2cm,levelsep=0.5cm]{\Tc*{0pt}}{\pstree[treesep=0.8cm,levelsep=0.5cm]{\Tc*{0pt}}{\Tc*{1.2pt} \Tc*{1.2pt} }  \Tc*{1.2pt} } } } 
}
\pscircle[linecolor=black](5.69,2.32){0.1}
\rput(5.51,2.6){\tiny$1$}
\rput(5.25,2.2){\tiny$2$}
\rput(6.1,2.2){\tiny$3$}
\rput(5.9,1.75){\tiny$4$}
\rput(5.3,1.15){\tiny$5$}
\rput(6.05,1.15){\tiny$6$}
\rput(6.65,1.75){\tiny$7$}
}

\rput(13,1.95){.}

\end{pspicture}
\end{center}
\end{figure} \vspace{-1.5cm}\hspace{-0.55cm}

\noindent Then we can define the following insertion operation on trees.

\begin{definition}
Let $m,n \in \mathbb{N}^*$ and $0<i\leq 2n-1$. Define the operation $\stackrel{\curvearrowright}{i} : \mathcal{PBT}(m) \times \mathcal{PBT}(n) \longrightarrow \mathcal{PBT}(n+m)$ for all $A\in \mathcal{PBT}(m)$ and $B \in \mathcal{PBT}(n)$  by 

\begin{figure}[H]

 \begin{center}
 \hspace{0.3cm}
\begin{pspicture}(0,0)(12.8,3.2)

\rput(0.2,1.8){$\stackrel{\curvearrowright}{i}$}
\rput(6.5,1.8){$=$}

\psscalebox{1.4}{
\rput(-0.8,1.2){
\pstree[treesep=0.5cm,levelsep=0.5cm]{\Tc*{0pt}}
{\pstree[treesep=1.2cm,levelsep=0.28cm]{\Tc*{1.2pt}}{\pstree{\Ttri{...}}{}}}
}
\pscircle[linecolor=black](-0.8,1.22){0.1}
}

\psscalebox{1.4}{
\rput(2.2,1.4){
\pstree[treesep=0.5cm,levelsep=0.5cm]{\Tc*{0pt}}
{\pstree[linestyle=dashed,treesep=0.6cm,levelsep=0.5cm]{\Tc*{1.2pt}}{\pstree[linestyle=solid,treesep=0.6cm,levelsep=0.5cm]{\Tc*{0pt}}{  \pstree[linestyle=solid,treesep=1.2cm,levelsep=0.5cm]{\Tc*{0pt}}{\Tfan[linestyle=solid]} \pstree[linestyle=solid,treesep=1.2cm,levelsep=0.5cm]{\Tc*{0pt}}{\Tfan[linestyle=solid]}}  \pstree[linestyle=solid,treesep=1.2cm,levelsep=0.5cm]{\Tc*{0pt}}{\Tfan[linestyle=solid]} }  }
}
\pscircle[linecolor=black](2.6,1.9){0.1}
\rput(1.3,1){\tiny$i$}
}

\psscalebox{1.4}{
\rput(5.5,1){
\pstree[treesep=0.5cm,levelsep=0.5cm]{\Tc*{0pt}}
{\pstree[treesep=1.2cm,levelsep=0.28cm]{\Tc*{1.2pt}}{\pstree{\Ttri{...}}{}}}
}
}

\psscalebox{1.4}{
\rput(8.5,1.4){
\pstree[treesep=0.5cm,levelsep=0.5cm]{\Tc*{0pt}}
{\pstree[linestyle=dashed,treesep=0.6cm,levelsep=0.5cm]{\Tc*{1.2pt}}{\pstree[linestyle=solid,treesep=0.6cm,levelsep=0.5cm]{\Tc*{0pt}}{  \pstree[linestyle=solid,treesep=1.2cm,levelsep=0.5cm]{\Tc*{0pt}}{\Tfan[linestyle=solid]} \pstree[linestyle=solid,treesep=1.2cm,levelsep=0.5cm]{\Tc*{0pt}}{\Tfan[linestyle=solid]}}  \pstree[linestyle=solid,treesep=1.2cm,levelsep=0.5cm]{\Tc*{0pt}}{\Tfan[linestyle=solid]} }  }
}
\pscircle[linecolor=black](8.9,1.9){0.1}
\rput(7.8,1){\tiny$i$}
}

\psbezier[linestyle=solid,linecolor=red]{-*}(7.56,2.1)(7.6,4)(7.8,1.5)(10.55,1.6)

\end{pspicture}
\end{center}
\end{figure} \vspace{-0.8cm}\hspace{-0.55cm}
that is the operation of insertion by the root edge of $A$ on the left of the $i^{\text{th}}$ edge in $B$. The root of $A \stackrel{\curvearrowright}{i} B$ is the root of $B$ if $i\neq 1$ or the newly created vertex if $i=1$.
\end{definition}

Matching the insertion operations on the trees and chord diagrams we obtain a correspondence defined recursively.

\begin{definition}
The map $\mathcal{T} : \mathcal{RCCD} \longrightarrow \mathcal{PBT}_{c}$, the set of rooted planar binary trees with leaves colored by the integers, is defined recursively on the chord diagrams labelled in the intersection order by:
\begin{figure}[H]
\begin{pspicture}(0,-1)(8,1)

\rput(-2,0){$\mathcal{T}$}\rput(-1.4,0){\huge $($ }
\rput(-2.1,-1.35){
\psscalebox{0.7}{
\pscircle(2.5,1.900){1.25}
\psline{*-*}(2.5,0.65)(2.5,3.15)
\pscircle[linecolor=red](2.5,0.65){0.12}
\rput(2.5,0.3){$i$} 
}
}
\rput(0.7,0){\huge $)$}
\rput(1,-0.1){$=$}

\rput(2,-0.2){
\pstree[treesep=0.5cm,levelsep=0.7cm]{\Tc*{0pt}}{\pstree[treesep=1.2cm,levelsep=0.5cm]{\Tc*{1.2pt}}{}}
}
\pscircle[linecolor=black](2,-0.3){0.12}
\rput(2,-0.6){\tiny$i$}

\rput(4,-0.1){and}

\rput(8,-0.1){$\mathcal{T}\left(X \stackrel{\curvearrowright}{(0,i)} Y \right) = \mathcal{T}(X) \stackrel{\curvearrowright}{i}  \mathcal{T}(Y)$\,.}

\end{pspicture}
\end{figure}

\end{definition}

It is very important that the $\mathcal{T}$-map is defined for rooted connected chord diagrams with chords labelled in the intersection order, and not the counterclockwise order, as it gives a simple characterisation of the image of $\mathcal{T}$. We will come back to that in the next section.\\

These are all the objects we need to introduce in order to prove the chord diagrams expansion giving the formal solution to the Dyson--Schwinger equation \eqref{DSEintro} however there are still many interesting combinatorial questions, like the statistical distribution of the sequences of gaps, that will need further study. \\

\subsection{Examples}

Next we give some examples of families of rooted connected chord diagrams for which all the objects introduced above are easily computed. 

\begin{example}
One of the simplest family of rooted connected chord diagrams consists in the cycloids. It is an example of family of diagrams which are minimally connected since removing any of the chords (except the first and the last one) result in a disconnected chord diagram. They are the diagrams such that the $i^{\text{th}}$ chord intersects only the $i+1^{\text{th}}$ chord:

\begin{figure}[H]
\begin{center}
\hspace{1.1cm}
\begin{pspicture}(0,0)(12.8,3.2)
\pscircle(-1.5,1.900){1.25} 
\psarc[linecolor=black]{*-*}(-0.7,0.7){1}{66}{181}
\pscircle[linecolor=red](-1.7,0.68){0.12}
\rput(-1.5,0){\small$(Cyc_{1})$}

\rput(0,1.900){,}

\pscircle(1.5,1.900){1.25}
\psarc[linecolor=black]{*-*}(2.3,0.7){1}{66}{181}
\psarc[linecolor=black]{*-*}(3.1,1.8){1}{126}{226}
\pscircle[linecolor=red](1.3,0.68){0.12} 
\rput(1.5,0){\small$(Cyc_{2})$}

\rput(3,1.900){,}

\pscircle(4.5,1.900){1.25} 
\psarc[linecolor=black]{*-*}(5.3,0.7){1}{66}{181}
\psarc[linecolor=black]{*-*}(6.1,1.8){1}{126}{226}
\psarc[linecolor=black]{*-*}(5.4,3.2){1}{184}{287}
\pscircle[linecolor=red](4.3,0.68){0.12} 
\rput(4.5,0){\small$(Cyc_{3})$}

\rput(6,1.900){,}

\pscircle(7.5,1.900){1.25} 
\psarc[linecolor=black]{*-*}(8.3,0.7){1}{66}{181}
\psarc[linecolor=black]{*-*}(9.1,1.8){1}{126}{226}
\psarc[linecolor=black]{*-*}(8.4,3.2){1}{184}{287}
\psarc[linecolor=black]{*-*}(6.8,3.2){1}{242}{354}
\pscircle[linecolor=red](7.3,0.68){0.12} 
\rput(7.5,0){\small$(Cyc_{4})$}

\rput(10,1.900){...}

\psarc[linecolor=black]{*-*}(13,0.5){0.8}{54}{166}
\psarc[linecolor=black]{*-*}(13.9,1.2){0.8}{102}{205}
\psarc[linestyle=dashed]{*-*}(13.9,2.4){0.8}{144}{256}
\psarc[linestyle=dashed]{*-*}(13,3.2){0.8}{188}{312}
\psarc[linestyle=dashed]{*-*}(11.8,3.2){0.8}{241}{355}
\psarc[linecolor=black]{*-*}(11,2.4){0.8}{-69}{32}
\psarc[linestyle=dashed]{-}(12.5,1.900){1.25}{25}{115}
\psarc[linecolor=black]{-}(12.5,1.900){1.25}{120}{379}
\pscircle[linecolor=red](12.23,0.69){0.12}
\rput(12.5,0){\small$(Cyc_{n})$}
\rput(12.2,0.45){\tiny$1$}
\rput(13.25,0.65){\tiny$2$}
\rput(13.6,2.9){\tiny$k$}
\rput(11.6,2.99){\tiny$n$}

\rput(14.2,1.900){...}

\end{pspicture}
\end{center}
\end{figure}

The directed intersection diagram of the cycloid with $n$ chords is simply the line:

\begin{figure}[H]
\begin{center}
\hspace{3.5cm}
\begin{pspicture}(0,0)(12.8,1)
\rput(-1,1){$\mathcal{I}(Cyc_{n})$}
\rput(0.8,1){1}
\psline{->}(1,1)(2,1)
\rput(2.25,1){2}
\psline{->}(2.5,1)(3.5,1)
\rput(3.75,1){3}
\psline{->}(4,1)(5,1)
\rput(5.9,1){ ... ... ...}
\psline{->}(7,1)(8,1)
\rput(8.25,1){$n$}

\rput(9,1){.}

\end{pspicture}
\vspace{-1.3cm}
\end{center}
\end{figure}

In particular its intersection order coincide with the counterclockwise order so we get the identity permutation $\sigma_{Cyc_{n}}=(123...n)$. Moreover there is only one terminal chord, Ter$_{\sigma}(Cyc_{n})=(n)$, which is then the smallest terminal chord of $(Cyc_{n})$. With only one terminal chord the sequence of gaps  is empty, $\delta(X)=\emptyset$. We associate to this chord diagram the sequence with length $n-1$:
\begin{equation*}
\bar{\delta}(Cyc_{n})=(0,0,\,\cdots,0)\,.
\end{equation*}

Finally, it is easy to compute the planar binary tree corresponding to the cycloid with $n$ chords. An easy induction gives that for all $n\in \mathbb{N}^*$ we have:

\vspace{0.2cm}
\begin{figure}[h!]
\begin{center}
\hspace{6cm}
\begin{pspicture}(0,-1)(8,1)
\rput(0,0){$\mathcal{T}(Cyc_{n})=$}
\psscalebox{1.2}{
\rput(3,0.2){
\pstree[treesep=0.5cm,levelsep=0.3cm]{\Tc*{0pt}}
{\pstree[treesep=1.2cm,levelsep=0.5cm]{\Tc*{1.2pt}}{\Tc*{1.2pt} \pstree[treesep=1.2cm,levelsep=0.5cm]{\Tc*{0pt}}{\Tc*{1.2pt} \psset{linestyle=dashed} \pstree[treesep=1.2cm,levelsep=0.5cm]{\Tc*{0pt}}{\Tc*{1.2pt}  \pstree[treesep=1.2cm,levelsep=0.5cm]{\Tc*{0pt}}{  \psset{linestyle=solid} \Tc*{1.2pt} \Tc*{1.2pt}}}}}}
}
\pscircle[linecolor=black](2.06,1.07){0.1}
\rput(1.44,0.4){\tiny$1$}
\rput(2.05,-0.1){\tiny$2$}
\rput(2.68,-0.6){\tiny$k$}
\rput(3.28,-1.13){\tiny$n-1$}
\rput(4.58,-1.13){\tiny$n$}
\rput(5.3,-0.1){.}
}
\end{pspicture}
\end{center}
\end{figure}

\end{example}

\begin{example}
Another very simple family of diagrams is given by the wheels with $n$ spokes. If the cycloids formed a family of minimally connected diagrams the wheel spokes are then maximally connected since removing any chord leaves us with a connected chord diagram.

\begin{figure}[H]
\begin{center}
\hspace{1.1cm}
\begin{pspicture}(0,0)(12.8,3.2)
\pscircle(-1.5,1.900){1.25} 
\psline{*-*}(-1.5,0.65)(-1.5,3.15)
\pscircle[linecolor=red](-1.5,0.65){0.12}
\rput(-1.5,0){\small$(W_{1})$}

\rput(0,1.900){,}

\pscircle(1.5,1.900){1.25}
\psline{*-*}(1.5,0.65)(1.5,3.15)
\psline{*-*}(0.25,1.900)(2.75,1.900)
\pscircle[linecolor=red](1.5,0.65){0.12}
\rput(1.5,0){\small$(W_{2})$}

\rput(3,1.900){,}

\pscircle(4.5,1.900){1.25} 
\psline{*-*}(4.5,0.65)(4.5,3.15)
\psrotate(4.5,1.9){45}{\psline{*-*}(3.25,1.900)(5.75,1.900)}
\psrotate(4.5,1.9){-45}{\psline{*-*}(3.25,1.900)(5.75,1.900)}
\pscircle[linecolor=red](4.5,0.65){0.12}
\rput(4.5,0){\small$(W_{3})$}

\rput(6,1.900){,}

\pscircle(7.5,1.900){1.25} 
\psline{*-*}(7.5,0.65)(7.5,3.15)
\psline{*-*}(6.25,1.900)(8.75,1.900)
\psrotate(7.5,1.9){45}{\psline{*-*}(6.25,1.900)(8.75,1.900)}
\psrotate(7.5,1.9){-45}{\psline{*-*}(6.25,1.900)(8.75,1.900)}
\pscircle[linecolor=red](7.5,0.65){0.12}
\rput(7.5,0){\small$(W_{4})$}

\rput(10,1.900){...}

\psarc[linestyle=dashed]{-}(12.5,1.900){1.25}{-25}{65}
\psarc[linestyle=dashed]{-}(12.5,1.900){1.25}{160}{230}
\psarc[linecolor=black]{-}(12.5,1.900){1.25}{60}{160}
\psarc[linecolor=black]{-}(12.5,1.900){1.25}{-123}{-30}
\pscircle[linecolor=red](12.5,0.65){0.12}
\psline{*-*}(12.5,0.65)(12.5,3.15)
\psrotate(12.5,1.9){35}{\psline[linestyle=dashed]{*-*}(11.25,1.900)(13.75,1.900)}
\psrotate(12.5,1.9){-60}{\psline[linestyle=solid]{*-*}(11.25,1.900)(13.75,1.900)}
\psrotate(12.5,1.9){60}{\psline[linestyle=solid]{*-*}(11.25,1.900)(13.75,1.900)}
\rput(12.5,0){\small$(W_{n})$}
\rput(12.5,0.41){\tiny$1$}
\rput(13.25,0.65){\tiny$2$}
\rput(13.71,2.75){\tiny$k$}
\rput(13.3,3.1){\tiny$n$}

\rput(14.2,1.900){...}

\end{pspicture}
\end{center}
\end{figure}

Any pair of chords intersects in $(W_{n})$ so its intersection diagram is a directed version of the complete graph on $n$ vertices:

\begin{figure}[H]
\begin{center}
\hspace{3cm}
\begin{pspicture}(0,0)(12.8,1.5)
\rput(-1,1){$\mathcal{I}(W_{n})$}
\rput(0.8,1){1}
\psline{->}(1,1)(2,1)
\psbezier{->}(0.85,1.15)(1.5,1.5)(1.5,1.5)(3.65,1.2)
\psbezier[linestyle=dashed]{->}(0.85,1.2)(2,1.8)(2,1.8)(5.9,1.1)
\psbezier[linestyle=solid]{->}(0.85,1.2)(2.1,2)(2.1,2)(8.25,1.2)
\psbezier[linestyle=dashed]{->}(2.28,0.8)(3.5,0.55)(3.5,0.55)(5.9,0.9)
\psbezier[linestyle=solid]{->}(2.28,0.8)(3.5,-0.15)(3.5,-0.15)(8.2,0.85)
\psbezier[linestyle=solid]{->}(3.72,0.8)(4.5,0.3)(4.5,0.3)(8.15,0.95)
\rput(2.25,1){2}
\psline{->}(2.5,1)(3.5,1)
\rput(3.75,1){3}
\psline{->}(4,1)(5,1)
\rput(5.9,1){ ... ... ...}
\psline{->}(7,1)(8,1)
\rput(8.25,1){$n$}

\rput(9,1){.}

\end{pspicture}
\end{center}
\end{figure}

Once more the counterclockwise order coincides with the intersection order and we get the identity permutation $\sigma_{W_{n}}=(123...n)$. Also there is only one terminal chord, Ter$_{\sigma}(W_{n})=( n )$ which is the smallest terminal chord of $W_{n}$. With only one terminal chord the sequence of gaps is empty $\delta(X)=\emptyset$ and we associate to this chord diagram the sequence with length $n-1$:
\begin{equation*}
\bar{\delta}(W_{n})=(0,0,\, \cdots ,0)\,.
\end{equation*} 

It is also easy to compute the planar binary tree corresponding to the wheel spoke with $n$ chords. An induction an the number of chords gives for all $n \in \mathbb{N}^*$ 

\vspace{0.2cm}
\begin{figure}[H]
\begin{center}
\hspace{6cm}
\begin{pspicture}(0,-1)(8,1)
\rput(0,0){$\mathcal{T}(W_{n})=$}
\psscalebox{1.2}{
\rput(3,0.2){
\pstree[treesep=0.5cm,levelsep=0.3cm]{\Tc*{0pt}}{\pstree[treesep=1.2cm,levelsep=0.5cm]{\Tc*{1.2pt}}{\pstree[treesep=1.2cm,levelsep=0.5cm]{\Tc*{0pt}}{ \pstree[treesep=1.2cm,levelsep=0.5cm]{ \psset{linestyle=dashed} \Tc*{0pt}}{\psset{linestyle=dashed} \pstree[treesep=1.2cm,levelsep=0.5cm]{\Tc*{0pt}}{\psset{linestyle=solid} \Tc*{1.2pt} \Tc*{1.2pt}} \Tc*{1.2pt}} \Tc*{1.2pt}} \Tc*{1.2pt}}}
}
\pscircle[linecolor=black](3.94,1.07){0.1}
\rput(1.44,-1.1){\tiny$1$}
\rput(2.74,-1.1){\tiny$2$}
\rput(3.32,-0.6){\tiny$k$}
\rput(4.01,-0.1){\tiny$n-1$}
\rput(4.58,0.4){\tiny$n$}
\rput(5.3,-0.1){.}
}
\end{pspicture}
\end{center}
\end{figure}

\end{example}

\begin{example}

We can also construct the family of ladders which maximize the number of terminal chords one can get for a connected chord diagram with a fixed number of chords: 

\begin{figure}[H]
\begin{center}
\hspace{1.1cm}
\begin{pspicture}(0,0)(12.8,3.2)
\pscircle(-1.5,1.900){1.25} 
\psline{*-*}(-1.5,0.65)(-1.5,3.15)
\pscircle[linecolor=red](-1.5,0.65){0.12}
\rput(-1.5,0){\small$(L_{1})$}

\rput(0,1.900){,}

\pscircle(1.5,1.900){1.25}
\psline{*-*}(1.5,0.65)(1.5,3.15)
\psline{*-*}(0.25,1.900)(2.75,1.900)
\pscircle[linecolor=red](1.5,0.65){0.12}
\rput(1.5,0){\small$(L_{2})$}

\rput(3,1.900){,}

\pscircle(4.5,1.900){1.25} 
\psline{*-*}(4.5,0.65)(4.5,3.15)
\psline{*-*}(3.32,1.5)(5.68,1.5)
\psline{*-*}(3.32,2.35)(5.68,2.35)
\pscircle[linecolor=red](4.5,0.65){0.12}
\rput(4.5,0){\small$(L_{3})$}

\rput(6,1.900){,}

\pscircle(7.5,1.900){1.25} 
\psline{*-*}(7.5,0.65)(7.5,3.15)
\psline{*-*}(6.25,1.900)(8.75,1.900)
\psline{*-*}(6.4,1.3)(8.61,1.3)
\psline{*-*}(6.4,2.5)(8.61,2.5)
\pscircle[linecolor=red](7.5,0.65){0.12}
\rput(7.5,0){\small$(L_{4})$}

\rput(10,1.900){...}

\psarc[linestyle=dashed]{-}(12.5,1.900){1.25}{140}{220}
\psarc[linestyle=dashed]{-}(12.5,1.900){1.25}{320}{400}
\psarc[linecolor=black]{-}(12.5,1.900){1.25}{40}{140}
\psarc[linecolor=black]{-}(12.5,1.900){1.25}{220}{320}
\pscircle[linecolor=red](12.5,0.65){0.12}
\psline{*-*}(12.5,0.65)(12.5,3.15)
\psline{*-*}(11.56,2.7)(13.45,2.7)
\psline[linestyle=dashed]{*-*}(11.28,2.1)(13.72,2.1)
\psline[linestyle=dashed]{*-*}(11.28,1.7)(13.72,1.7)
\psline{*-*}(11.56,1.1)(13.45,1.1)
\rput(12.5,0){\small$(L_{n})$}
\rput(12.5,0.39){\tiny$1$}
\rput(13.5,0.90){\tiny$2$}
\rput(13.9,2.1){\tiny$k$}
\rput(13.65,2.8){\tiny$n$}

\rput(14.2,1.900){...}

\end{pspicture}
\end{center}
\end{figure}

In $L_{n}$ the first chord intersects all the other ones so its intersection diagram has one edge $(1,k)$ for each $k \in \lbb 2,n \rbb$:

\begin{figure}[H]
\begin{center}
\hspace{3cm}
\begin{pspicture}(0,0)(12.8,1.5)
\rput(-1,1){$\mathcal{I}(L_{n})$}
\rput(0.8,1){1}
\psline{->}(1,1)(2,1)
\psbezier{->}(0.85,1.1)(2.5,1.5)(2.5,1.5)(3.75,1.2)
\rput(2.25,1){2}
\psbezier{->}(0.8,1.2)(3,1.8)(3,1.8)(6.5,1.2)
\rput(3.75,1){3}
\rput(4.9,1){ ... ... }
\psbezier[linestyle=dashed]{->}(0.95,0.85)(2.5,0.5)(2.5,0.5)(4.9,0.8)
\rput(6.5,1){$n-1$}
\psbezier{->}(0.85,0.8)(3,0.1)(3,0.1)(8.25,0.8)
\rput(8.25,1){$n$}

\rput(9,1){.}

\end{pspicture}
\end{center}
\end{figure}

So in $L_{n}$ we have $n-1$ terminal chords, Ter$(L_{n})=\{2,3,...,n\}$ and the gap between any two consecutive chords is $1$ hence $\delta(L_{n})=(1,1,...,1)$. Once more the counterclockwise order coincide with the intersection order so $\sigma_{L_{n}}=(123....n)$.\\

We can compute inductively the planar binary tree corresponding to $L_{n}$. For all $n \in \mathbb{N}^*$ we have:

\vspace{0.2cm}
\begin{figure}[H]
\begin{center}
\hspace{6cm}
\begin{pspicture}(0,-1)(8,1)
\rput(0,0){$\mathcal{T}(L_{n})=$}
\psscalebox{1.2}{
\rput(3,0.2){
\pstree[treesep=0.5cm,levelsep=0.3cm]{\Tc*{0pt}}{\pstree[treesep=1.2cm,levelsep=0.5cm]{\Tc*{1.2pt}}{\pstree[treesep=1.2cm,levelsep=0.5cm]{\Tc*{0pt}}{ \pstree[treesep=1.2cm,levelsep=0.5cm]{ \psset{linestyle=dashed} \Tc*{0pt}}{\psset{linestyle=dashed} \pstree[treesep=1.2cm,levelsep=0.5cm]{\Tc*{0pt}}{\psset{linestyle=solid} \Tc*{1.2pt} \Tc*{1.2pt}} \Tc*{1.2pt}} \Tc*{1.2pt}} \Tc*{1.2pt}}}
}
\pscircle[linecolor=black](3.94,1.07){0.1}
\rput(1.44,-1.1){\tiny$1$}
\rput(2.74,-1.1){\tiny$n$}
\rput(3.32,-0.6){\tiny$k$}
\rput(4.01,-0.1){\tiny$3$}
\rput(4.58,0.4){\tiny$2$}
\rput(5.3,-0.1){.}
}
\end{pspicture}
\end{center}
\end{figure}

\end{example}

\begin{example}
Here is a family of rooted connected chord diagrams that shows we can get any sequence of gaps and also illustrates the difference between the counterclockwise order and the intersection order:

\begin{figure}[H]

 \begin{center}
 \hspace{3.0cm}
\begin{pspicture}(0,0)(12.8,3.5)
\psscalebox{1}{
\pscircle(1,1.900){2} 
\psrotate(1,1.9){90}{\psline{*-*}(-1,1.9)(3,1.9)}
\psrotate(1,1.9){160}{\psline{*-*}(-1,1.9)(3,1.9)}
\psrotate(1,1.9){205}{\psline{*-*}(-1,1.9)(3,1.9)}
\psarc[linecolor=black]{*-*}(0.7,4.3){0.9}{225}{330} 
\psarc[linecolor=black]{*-*}(-0.3,4){0.8}{255}{348} 
\psarc[linecolor=black]{*-*}(-1.2,2.4){0.6}{-70}{45} 
\psarc[linecolor=black]{*-*}(-1,0.8){0.6}{-27}{86} 
\psarc[linecolor=black]{*-*}(-0.65,0.22){0.6}{-4}{94} 
\psrotate(1,1.9){20}{\psarc[linecolor=black]{*-*}(-0.65,0.22){0.6}{-4}{94} }
\psrotate(1,1.9){90}{\pscircle[linecolor=red](-1,1.9){0.12}}
\rput(1,-0.35){\tiny$1$}
\rput(3,1.05){\tiny$2$}
\rput(3.05,2.9){\tiny$3$}
\rput(1.55,4){\tiny$4$}
\rput(0.35,4){\tiny$5$}
\rput(-1,2.9){\tiny$6$}
\rput(-1.15,1.45){\tiny$7$}
\rput(-0.9,0.8){\tiny$8$}
\rput(-0.4,0.2){\tiny$9$}
\rput(1,-0.95){\small$CW_{3}(2,1,3)$}

\psarc[linestyle=solid]{-}(8,1.9){2}{205}{335}
\psarc[linestyle=dashed]{-}(8,1.9){2}{-25}{35}
\psarc[linestyle=solid]{-}(8,1.9){2}{39}{160}
\psarc[linestyle=dashed]{-}(8,1.9){2}{163}{203}
\psline{*-*}(8,-0.1)(8,3.9)
\psrotate(8,1.9){55}{\psline{*-*}(8,-0.1)(8,3.9)}
\psrotate(8,1.9){135}{\psline{*-*}(8,-0.1)(8,3.9)}
\psarc[linestyle=dashed,linecolor=blue]{*-*}(7.8,4.4){0.8}{230}{320}
\psarc[linestyle=dashed,linecolor=blue]{*-*}(7.1,4.1){0.55}{250}{335}
\psrotate(8,1.9){60}{
\psarc[linestyle=dashed,linecolor=blue]{*-*}(7.8,4.4){0.8}{230}{320}
\psarc[linestyle=dashed,linecolor=blue]{*-*}(7.1,4.1){0.55}{250}{335}
}
\psrotate(8,1.9){135}{
\psarc[linestyle=dashed,linecolor=blue]{*-*}(7.8,4.4){0.8}{230}{320}
\psarc[linestyle=dashed,linecolor=blue]{*-*}(7.1,4.1){0.55}{250}{335}
}
\pscircle[linecolor=red](8,-0.1){0.12}
\rput(8,-0.35){\tiny$1$}
\rput(9.7,0.55){\tiny$2$}
\rput(9.6,3.45){\tiny$n$}
\rput(7.5,4.15){\tiny$B_{1}$}
\rput(5.75,2.45){\tiny$B_{2}$}
\rput(6.75,-0.05){\tiny$B_{n}$}

\rput(8,-0.95){\small$CW_{n}(\beta_{1},\beta_{2},...,\beta_{n})$}

\rput(12,1.9){.}

}
\end{pspicture}
\end{center}
\vspace{0.5cm}

\end{figure} \hspace{0.5cm}

Here $CW_{n}(\beta_{1},...,\beta_{n})$ is formed by a wheel spoke with $n$ chords such that the target of the $k^{\text{th}}$ spoke lies in the first interval of a block $B_{k}$ consisting of cycloid with $\beta_{k}$ chords.

The intersection diagram of $CW_{n}(\beta_{1},...,\beta_{n})$ consists of two main parts mixing what we already know for the wheels and for the cycloids. First, the vertices from $1$ to $n$ form a directed version of the complete graph on $n$ vertices then we have $n$ disjoint lines corresponding to the cycloids $B_{1}$ to $B_{n}$ and for each $k \in \lbb 1,n \rbb$ an edge directed from the $k^{\text{th}}$ vertex to the initial vertex of the block $\mathcal{I}(B_{k})$.\\




The terminal chords of $CW_{n}(\beta_{1},...,\beta_{n})$ are the last chords of each of the $n$ cycloids. This time the intersection order does not corresponds to the counterclockwise order. The first $n$ chords are in a wheel spokes configuration so their counterclockwise order is preserved. Inside each block $B_{k}$, since we have cycloids, the counterclockwise order is preserved. However the first spoke lies in the block $B_{1}$, so in the intersection order this is the last block. The second spoke lies in $B_{k}$ which becomes the next to last block, and so on. We end up with the permutation:
\begin{equation*}
\sigma_{CW_{n}}=(1...n\,B_{n}\,...B_{2}\,B_{1})\,.
\end{equation*}
\noindent In particular we obtain $CW_{3}(2,1,3)$ in the intersection order:

\begin{figure}[H]

 \begin{center}
 \hspace{3.0cm}
\begin{pspicture}(-1,0)(12.8,3.5)

\rput(5,0){
\psscalebox{1}{
\pscircle(1,1.900){2} 
\psrotate(1,1.9){90}{\psline{*-*}(-1,1.9)(3,1.9)}
\psrotate(1,1.9){160}{\psline{*-*}(-1,1.9)(3,1.9)}
\psrotate(1,1.9){205}{\psline{*-*}(-1,1.9)(3,1.9)}
\psarc[linecolor=black]{*-*}(0.7,4.3){0.9}{225}{330} 
\psarc[linecolor=black]{*-*}(-0.3,4){0.8}{255}{348} 
\psarc[linecolor=black]{*-*}(-1.2,2.4){0.6}{-70}{45} 
\psarc[linecolor=black]{*-*}(-1,0.8){0.6}{-27}{86} 
\psarc[linecolor=black]{*-*}(-0.65,0.22){0.6}{-4}{94} 
\psrotate(1,1.9){20}{\psarc[linecolor=black]{*-*}(-0.65,0.22){0.6}{-4}{94} }
\psrotate(1,1.9){90}{\pscircle[linecolor=red](-1,1.9){0.12}}
\rput(1,-0.35){\tiny$1$}
\rput(3,1.05){\tiny$2$}
\rput(3.05,2.9){\tiny$3$}
\rput(1.55,4){\tiny$8$}
\rput(0.35,4){\tiny$9$}
\rput(-1,2.9){\tiny$7$}
\rput(-1.15,1.45){\tiny$4$}
\rput(-0.9,0.8){\tiny$5$}
\rput(-0.4,0.2){\tiny$6$}
\rput(1,-0.95){\small$CW_{3}(2,1,3)$}

\rput(7,1.9){.}

}
}
\end{pspicture}
\end{center}
\vspace{0.5cm}

\end{figure} \hspace{0.5cm}

\noindent The sequence of gaps is then the list of the differences between the consecutive terminal chords of the cycloid blocks in their intersection order:
\begin{equation*}
\delta(CW_{n}(\beta_{1},...,\beta_{n}))=\left(\sigma_{CW_{n}}(\beta_{n-1})-\sigma_{CW_{n}}(\beta_{n}],...,\sigma_{CW_{n}}(\beta_{1})-\sigma_{CW_{n}}(\beta_{2}) \right)\,.
\end{equation*}

\noindent A simple induction gives the corresponding planar binary tree. For all $n \in \mathbb{N}^*$ we have:

\vspace{0.7cm}
\begin{figure}[H]
\begin{center}
\hspace{6cm}
\begin{pspicture}(0,-1.5)(8,1)
\rput(-2.5,0){
\rput(0,0){$\mathcal{T}(W_{n})=$}
\psscalebox{1.1}{
\rput(3.5,0.25){
\pstree[treesep=0.5cm,levelsep=0.3cm]{\Tc*{0pt}}{\pstree[treesep=0.5cm,levelsep=0.8cm]{\Tc*{1.2pt}}{\pstree[treesep=0.5cm,levelsep=0.8cm]{\Tc*{0pt}}{ \pstree[treesep=0.5cm,levelsep=0.8cm]{ \psset{linestyle=dashed} \Tc*{0pt}}{\psset{linestyle=dashed} \pstree[treesep=0.5cm,levelsep=0.5cm]{\Tc*{0pt}}{\psset{linestyle=solid} \Tc*{1.2pt} \Tr{\tiny$\mathcal{T}(B_{1})$}} \pstree[treesep=0.5cm,levelsep=0.5cm]{\Tc*{1.2pt}}{\Tc*{1.2pt} \Tr{\tiny$\mathcal{T}(B_{k})$}}} \pstree[treesep=0.5cm,levelsep=0.5cm]{\Tc*{1.2pt}}{\Tc*{1.2pt} \Tr{\tiny$\mathcal{T}(B_{n-1})$}}} \pstree[treesep=0.5cm,levelsep=0.5cm]{\Tc*{1.2pt}}{\Tc*{1.2pt} \Tr{\tiny$\mathcal{T}(B_{n})$}}}}
}
\pscircle[linecolor=black](4.3,1.6){0.1}
\rput(0.79,-1.5){\tiny$1$}
\rput(2.65,-1.5){\tiny$k$}
\rput(3.8,-0.65){\tiny$n-1$}
\rput(4.88,0.1){\tiny$n$}
\rput(7.3,-0.1){,}
}
}
\end{pspicture}
\end{center}
\end{figure}
\noindent where $\mathcal{T}(B_{k})$ corresponds to the tree of the $k^{th}$ cycloid block.

\end{example}

\section{Recurrences}\label{rec section}

In this section we will use the constructions above to prove some technical lemmas and then the main theorem.

\subsection{A refinement of a recurrence of Stein}

Let $c_n = |\mathcal{RCCD}(n)|$.
A classical recurrence for $c_n$ is 
\[
c_n = (n-1)\sum_{k=1}^{n-1}c_kc_{n-k} \quad \text{for $n\geq 2$} \qquad c_1=1
\]
due to Stein \cite{Schord}.  Nijenhuis and Wilf \cite{NWchord} give a proof of this recurrence using the root-share decomposition.  What their proof naturally gives is the equivalent recurrence
\[
  c_n = \sum_{k=1}^{n-1}(2k-1)c_k c_{n-k} \quad \text{for $n\geq 2$} \qquad c_1=1
\]

 By following essentially the same proof while keeping track of the terminal chords we get the following 
\begin{prop}\label{gamma k in g}
Let
\[
  g_i = \sum_{\substack{C\in\mathcal{RCCD}\\b(C)\geq i}}x^{|C|}f_{C}f_{b(C)-i}
\]
then
\[
   g_k = g_1 \left(2x\frac{d}{dx}-1\right)g_{k-1} \quad \text{for $k\geq 2$} 
\]
\end{prop}

Before proving the recurrence lets see how it refines the classic chord diagram recurrence.  Note that the lowest power of $x$ appearing in $g_k$ is $k$.  Write
\[
  g_k = \sum_{i\geq k}g_{k,i}x^i
\]
so
\[
g_{k,i} = \sum_{\substack{C\\|C|=i\\b(C)\geq i}}f_{C}f_{b(C)-i}
\]
Then Proposition \ref{gamma k in g} is the statement
\begin{equation}\label{g rec spread out}
  g_{k,i} = \sum_{\ell=1}^{i-1}(2\ell-1)g_{1,i-\ell}g_{k-1,\ell} \quad \text{for $2\leq k\leq i$}
\end{equation}
which has the same form as the classic chord diagram recurrence, but with the $g_{k,i}$ rather than simple counts.

The key to the proof is the behaviour of $\bar{\delta}$ in the root-share decomposition.  This is encapsulated in the following lemma.

\begin{lemma}\label{first rec claim}

  Let $C_1$ and $C_2$ be rooted connected chord diagrams, and take $1 \leq m \leq 2|C_2|-1$.
  Let
  \[
  C=C_1 \stackrel{\curvearrowright}{(0,m)}C_2.
  \]
Then for any $k$ for which all terms are defined we have
\[
f_{C}f_{b(C)-k} = f_{C_1}f_{b(C_1)-1}f_{C_2}f_{b(C_2)-k+1}
\]
\end{lemma}

\begin{proof}
The terminal chords of $C$ are the terminal chords of $C_1$ along with the terminal chords of $C_2$.  Furthermore, if $c$ is a chord in $C_2$ with index $i$ in the intersection order, then $c$ has index $i+1$ in the intersection order in $C$, and if $c$ is a terminal chord in $C_1$ with index $j$ in the intersection order then $c$ has index $j+|C_2|$ in the intersection order in $C$.  This gives that 
\[ 
b(C) = b(C_2)+1.
\]
Also the last terminal chord of $C_2$ is the last chord of $C_2$ which has index $|C_2|+1$ in $C$ and the next terminal chord in $C$ is the first terminal chord of $C_1$ which has index $b(C_1)+|C_2|$ in $C$.

So $\bar{\delta}(C)$ is the concatenation of $\bar{\delta}(C_2)$, $b(C_1)+|C_2|-(|C_2|+1)-1$, and $\bar{\delta}(C_1)$.  Simplifying
\[
\bar{\delta}(C) = (\bar{\delta}(C_2),b(C_1)-2,\bar{\delta}(C_1))
\]
The result follows.
\end{proof}

\begin{proof}[Proof of Proposition \ref{gamma k in g}]
  Take $k \geq 2$.

  Take two rooted connected chord diagrams $C_1$ and $C_2$, where $|C_2|=\ell$.  Take $1 \leq m \leq 2\ell-1$ and 
  let
  \[
  C=C_1 \stackrel{\curvearrowright}{(0,m)}C_2
  \]
By Lemma \ref{first rec claim}
\[
f_{C}f_{b(C)-k} = f_{C_1}f_{b(C_1)-1}f_{C_2}f_{b(C_2)-k+1}
\]
so the term contributed to the right hand side of \eqref{g rec spread out} by $C$ is the same as the term contributed to the left hand side of \eqref{g rec spread out} by $C_1$ and $C_2$.
Since the root share decomposition is unique and there are $2\ell-1$ choices for $m$, we thus obtain \eqref{g rec spread out}
which gives the proposition.
\end{proof}

Scaling to match what we will ultimately need, if we let
\[
  \gamma_i = \frac{(-1)^i}{i!}g_i
\]
then Proposition \ref{gamma k in g} gives
\begin{equation}\label{gamma k rec rephrased}
\gamma_k(x) = \frac{1}{k}\gamma_1(x)\left(-1+2x\frac{d}{dx}\right)\gamma_{k-1}(x
) \quad \text{for $k\geq 2$}.
\end{equation}


\subsection{A tree recurrence}

The goal of this subsection is the following somewhat technical recurrence.
\begin{prop}\label{second rec}
\begin{equation}\label{key lemma}
\sum_{\substack{C \in\mathcal{RCCD}\\|C|=i+1\\b(C) = j+1}}f_{C} = \sum_{k=1}^{i}\sum_{\ell = 1}^{j}\binom{j}{\ell}\left(\sum_{\substack{C\in\mathcal{RCCD}\\|C|=k\\b(C)\geq \ell}}f_{C}f_{b(C)-\ell}\right)\left(\sum_{\substack{C\in\mathcal{RCCD}\\|C|=i-k+1\\b(C)=j-\ell+1}}f_{C}\right)
\end{equation}
for $i \geq 1$ and $j \geq 1$.
\end{prop}

This recurrence comes from the decomposition of $\mathcal{T}(C)$ into the left and right subtrees of the root.  This is a very natural decomposition at the level of trees, but it is not at all apparent at the level of chord diagrams as can be seen on the following examples where the chords colored in red are in the right subtree, the chords in blue are in the left subtree and the subtree of the root-share have dotted edges :\\

\begin{figure}[H]

 \begin{center}
 \hspace{3cm}
\begin{pspicture}(0,0)(12.8,3.2)

\psrotate(-2,1.9){130}{
\pscircle(-2,1.900){1.25}
\psline[linecolor=red]{*-*}(-2,0.65)(-2,3.15)
\psarc[linecolor=red]{*-*}(-2,3.15){1}{204}{336}
\psrotate(-2,1.9){180}{\psarc[linecolor=blue]{*-*}(-2,3.15){1}{204}{336}}
}
\pscircle[linecolor=red](-2.175,0.65){0.12}
\rput(-2.175,0.4){\tiny$1$}
\rput(-0.725,1.75){\tiny$2$}
\rput(-1.05,2.85){\tiny$3$}

\rput(2.2,2){
\pstree[treesep=0.5cm,levelsep=0.5cm]{\Tc*{0pt}}{\pstree[treesep=1.2cm,levelsep=0.7cm]{\Tc*{1.2pt}}{\Tc*{1.2pt} \pstree[treesep=1.2cm,levelsep=0.7cm]{\Tc*{1.2pt}}{\psset{linestyle=dotted,dotsep=1pt} \Tc*{1.2pt} \psset{linestyle=solid} \Tc*{1.2pt}} }}
}
\pscircle[linecolor=black](1.885,2.45){0.12}
\rput(1.89,0.9){\tiny$1$}
\rput(3.165,0.9){\tiny$3$}
\rput(1.25,1.6){\tiny$2$}

\pscircle(7,1.900){1.25}
\psline[linecolor=blue]{*-*}(7,0.65)(7,3.15)
\psrotate(7,1.9){60}{\psline[linecolor=blue]{*-*}(7,0.65)(7,3.15)}
\psrotate(7,1.9){120}{\psline[linecolor=red]{*-*}(7,0.65)(7,3.15)}
\pscircle[linecolor=red](7,0.65){0.12}
\rput(7,0.4){\tiny$1$}
\rput(8.3,1.2){\tiny$2$}
\rput(8.25,2.65){\tiny$3$}

\rput(11.8,2){
\pstree[treesep=0.5cm,levelsep=0.5cm]{\Tc*{0pt}}{\pstree[treesep=1.2cm,levelsep=0.7cm]{\Tc*{1.2pt}}{\pstree[treesep=1.2cm,levelsep=0.7cm]{\Tc*{1.2pt}}{\psset{linestyle=dotted,dotsep=1pt} \Tc*{1.2pt} \psset{linestyle=solid} \Tc*{1.2pt}} \Tc*{1.2pt}}}
}
\pscircle(12.125,2.45){0.12}
\rput(10.85,0.9){\tiny$1$}
\rput(12.15,0.9){\tiny$2$}
\rput(12.825,1.6){\tiny$3$}

\rput(13.825,1.6){.}

\end{pspicture}
\end{center}

\end{figure} \hspace{0.5cm}

In order to prove Proposition \ref{second rec} we need to characterize $\mathcal{T}(\mathcal{RCCD})$ and then we need some technical lemmas. 

\begin{definition} The branch of a planar binary tree which follows the right child at each vertex will be called the \textbf{fully right branch}. The number of leaves of a tree $t$ will be denoted $\ell(t)$
\end{definition}

The elements of $\mathcal{T}(\mathcal{RCCD})$ are built inductively by insertions.  These insertions shift the labels of the leaves in the intersection order in a very simple manner. This is what the following property encodes.

\begin{lemma}
Trees in $\mathcal{T}(\mathcal{RCCD})$ have the following property:
\begin{itemize}
\item[P1:] At any vertex $w$, the smallest label in the left subtree of $w$ is smaller than the label at the end of the fully right branch from $w$.
\end{itemize}
\end{lemma}

\begin{proof}
After we have made a critical observation concerning the properties of the intersection order, the proof follows from an elementary induction.\\

Let $X$ and $Y$ be elements of $\mathcal{RCCD}$ whose sequence of chords labelled in the intersection order are respectively $(\sigma_{i_{1}},\cdots,\sigma_{i_{n}})$ and $(\pi_{i_{1}},\cdots,\pi_{i_{m}})$. Then for any $0 < k \leq 2m-1$, the chords of $X \stackrel{\curvearrowright}{(0,k)} Y $ in the intersection order are labelled:
\begin{equation*}
\sigma_{i_{1}} < \pi_{i_{1}} + 1 < \cdots < \pi_{i_{m}}+1 < \sigma_{i_{2}}+m < \cdots < \sigma_{i_{n}} + m 
\end{equation*}
since the connected component attached to the root coming from $X$ comes last in the definition of the intersection order.\\

The rest follows immediately by induction on the number of leaves of our trees. The base case is obvious:

\begin{figure}[H]

 \begin{center}
 \vspace{1cm}
\begin{pspicture}(0,0)(12.8,1)

\rput(7,2){
\pstree[treesep=0.5cm,levelsep=0.5cm]{\Tc*{0pt}}{\pstree[treesep=1.2cm,levelsep=0.7cm]{\Tc*{1.2pt}}{\Tc*{1.2pt} \Tc*{1.2pt}}}
}
\pscircle(7,2.1){0.12}
\rput(6.35,1.2){\tiny$1$}
\rput(7.65,1.2){\tiny$2$}

\end{pspicture}
\end{center}

\end{figure} \vspace{-1.2cm}

So assume that any element of $\mathcal{T}(\mathcal{RCCD})$ with at most $n-1$ leaves has the desired property and now consider a tree $T$ with $n$ leaves.

Because of the root-share decomposition there is some integer $k$ so that we can write $T=T_{0}\stackrel{\curvearrowright}{k} T_{1} $ with $T_{0},T_{1} \in \mathcal{T}(\mathcal{RCCD})$ having less than $n$ leaves. There are three possibilities:
\begin{enumerate}
\item if $w$ is the insertion vertex in $T$ then the left subtree must contain the label $1$ and the property is true;
\item if $w$ is a vertex of $T$ that was a vertex of $T_{0}$ then by the induction hypothesis the property must hold since by the preliminary remark the labels in the subtree at $w$ have all been shifted by $\ell(T_1)$;
\item if $w$ is a vertex of $T$ that was a vertex of $T_{1}$ then by the induction hypothesis the property must hold since by the preliminary remark the labels in the subtree at $w$ have all been shifted by $1$;
\end{enumerate}

So the property holds at any vertex of $T$ which concludes the induction.

\end{proof}

\begin{definition}
Given $T\in\mathcal{PBT}_{c}$, by \textbf{removing} the subtree rooted at a vertex $w$ we mean deleting this subtree and replacing the parent of $w$ with its other child.

Given $T\in\mathcal{PBT}_{c}$  which satisfies P1, define the \textbf{smallest removable subtree containing 1} to be the minimal subtree rooted at some $w$ of $T$ which has $1$ as a leaf and which can be removed while maintaining P1. 
\end{definition}

\noindent Here are some trees where the smallest removable subtrees have been marked in red:

\begin{figure}[H]

 \begin{center}
 \vspace{0.5cm}
 \hspace{1cm}
\begin{pspicture}(0,0)(12.8,3.2)

\psscalebox{1.4}{
\rput(-0.5,1.8){
\pstree[treesep=0.5cm,levelsep=0.5cm]{\Tc*{0pt}}
{\pstree[treesep=1.2cm,levelsep=0.5cm]{\Tc*{1.2pt}}{ \pstree[treesep=1.2cm,levelsep=0.5cm,linecolor=red]{\Tc*{0pt}}{\Tc*{1.2pt}  \Tc*{1.2pt} } \Tc*{1.2pt} } }
}
\pscircle[linecolor=black](-0.18,2.07){0.1}
\rput(-1.45,0.9){\tiny$1$}
\rput(-0.15,0.9){\tiny$3$}
\rput(0.45,1.4){\tiny$2$}
}

\psscalebox{1.4}{
\rput(3.5,1.8){
\pstree[treesep=0.5cm,levelsep=0.5cm]{\Tc*{0pt}}
{\pstree[treesep=1.2cm,levelsep=0.5cm]{\Tc*{1.2pt}}{\Tc*{1.2pt} \pstree[treesep=1.2cm,levelsep=0.5cm]{\Tc*{0pt}}{\pstree[treesep=0.8cm,levelsep=0.5cm]{\Tc*{0pt}}{\psset{linecolor=red} \Tc*{1.2pt} \psset{linecolor=black} \Tc*{1.2pt} }  \Tc*{1.2pt} } } } 
}
\pscircle[linecolor=black](3.19,2.32){0.1}
\rput(2.58,1.65){\tiny$3$}
\rput(2.75,0.65){\tiny$1$}
\rput(3.65,0.65){\tiny$2$}
\rput(4.45,1.15){\tiny$4$}
}

\psscalebox{1.4}{
\rput(7.7,1.8){
\pstree[treesep=0.5cm,levelsep=0.5cm]{\Tc*{0pt}}
{\pstree[treesep=1cm,levelsep=0.7cm]{\Tc*{1.2pt}}{ \pstree[treesep=1cm,levelsep=0.5cm,linecolor=red]{\Tc*{0pt}}{\Tc*{1.2pt}  \Tc*{1.2pt} } \pstree[treesep=1cm,levelsep=0.5cm]{\Tc*{0pt}}{\Tc*{1.2pt}  \Tc*{1.2pt} } } }
}
\pscircle[linecolor=black](7.7,2.17){0.1}
\rput(6.07,0.8){\tiny$1$}
\rput(7.15,0.8){\tiny$4$}
\rput(8.25,0.8){\tiny$2$}
\rput(9.33,0.8){\tiny$3$}
}

\rput(14,1.95){.}

\end{pspicture}
\end{center}
\end{figure} \vspace{-1.2cm}\hspace{-0.55cm}

\begin{lemma}
A tree $T$ in $\mathcal{T}(\mathcal{RCCD})$ has the following property:
\begin{itemize}
\item[P2:] Let $H$ be the smallest removable subtree of $T$ containing $1$.  $H$ has leaf labels $1$, $\ell(T)-\ell(H)+2, \ell(T)-\ell(H)+3, \ldots, \ell(T)$. 
\end{itemize}
\end{lemma}

\begin{proof}
  Let $T$, $T_1$, $T_2$ as in the previous proof, $v$ be a vertex of $T$ and $C$, $C_1$, $C_2$ the chord diagrams corresponding to $T$, $T_1$, $T_2$ respectively.

  $H$ can not be strictly contained in $T_2$ as if it were then when $H$ is removed the leaves of the left child of $v$ are labelled with some labels from $|C_2|+1$ to $|C|$ while all other labels in the tree are at most $|C_2|$.  Thus P1 is not satisfied.

  $T_2$ can not be strictly contained in $H$ as $T_2$ is a removable subtree of $T$ by construction.

  Thus $T_2=H$ and as a consequence $H$ has P2.
\end{proof}

Say a tree satisfies $P2$ recursively if it satisfies $P2$ and what remains after removing the largest removable subtree containing $1$ satisfies $P2$ and so on.

\begin{thm}
  $\mathcal{T}(\mathcal{RCCD})$ is the set of $T\in \mathcal{PBT}_{c}$
which satisfy $P1$ and satisfy $P2$ recursively.
\end{thm}

\begin{proof}
By the preceding lemmas every $T \in \mathcal{T}(\mathcal{RCCD})$ satisfies $P1$ and satisfies $P2$ recursively.

To show that these properties characterize the image, we will inductively define the inverse of $\mathcal{T}$.  Map the tree with only one vertex to the rooted chord diagram with only one chord.
Consider $T\in \mathcal{PBT}_{c}$ with at least 2 vertices and
which satisfies $P1$ and satisfies $P2$ recursively.
Let $T_2$ be the largest removable subtree of $T$ containing $1$.  Let $T_1$ be $T$ with $T_2$ removed.  Let $k$ be the index of the vertex where $T_2$ was inserted.
Shift the vertex labels of $T_1$ and $T_2$ preserving their orders.  Inductively associate $T_1$ and $T_2$ to chord diagrams $C_1$ and $C_2$.  Associate $T$ to $C_1  \stackrel{\curvearrowright}{(0,k)} C_2$. 

This construction is clearly inverse to $\mathcal{T}$.
\end{proof}

Next we need one further nice observation and three technical lemmas.

\begin{prop}
  Let $C$ be a connected rooted chord diagram.  The $b(C)$ is the label of the leaf of the fully right branch of $\mathcal{T}(C)$.
\end{prop}

\begin{proof}
  The result is clear if $|C|=1$.

  Suppose $|C|>1$.  Let $C_1$, $C_2$, $T$, $T_1$, and $T_2$ be as in the above proofs.  By construction $b(C) = b(C_2)+1$.  The insertion of $T_2$ into $T_1$ does not affect the fully right branch except to shift the labels of $T_2$ up by 1.   By induction the label of the leaf of the fully right branch of $T_2$ is $b(C_2)$ and so the label of the leaf of the fully right branch of $T$ is $b(C_2)+1=b(C)$.
\end{proof}

\begin{lemma}\label{old claim 1}
  Let $C\in\mathcal{RCCD}$ with $|C|\geq 2$.  Let $T=\mathcal{T}(C)$, and let $H_1$ and $H_2$ be the left and right subtrees, respectively, of $T$.  Suppose $D_1$ and $D_2$ are chord diagrams with $\mathcal{T}(D_i) = H_i$.  Then
 \[b(D_1) \geq b(C)-b(D_2).\]
\end{lemma}

Another way to say this is that the leaf label of the fully right branch of $H_1$, with labels shifted to be consecutive, is $\geq b(C)-b(D_2)$.  That is, this lemma tells us that  $D_1$ and $D_2$ decompose $C$ into pieces which satisfy the correct conditions to be from a term on the right hand side of \eqref{key lemma}.

\begin{proof}
  Let $b(C) = j+1$ and let $b(D_2)= j-\ell+1$, so $\ell=b(C)-b(D_2)$.  $b(C)$ is the label of the leaf of the fully right branch of $T$ and $b(D_2)$ is the label of the leaf of the fully right branch of $H_2$, but these are the same leaf, the only difference is the label shifting.   Thus there must be $\ell$ labels on the $H_1$ side of $t$ which are less than $b(C)$. 
  
  Let $C=C_1 \stackrel{\curvearrowright}{(0,k)} C_2$ be the root-share decomposition of $C$.  Let $T_i = \mathcal{T}(C_i)$, and let $v$ be the $k$th vertex of $T$, that is the insertion vertex.
The proof of the lemma 
is an induction on the size of $C_2$.

If $|C_2|=1$ then $C_2=D_2$ and $b(C)=2$.  Also $b(D_2)=1$, so $\ell=1$, and since $1$ is the smallest label $b(D_1) \geq 1 = \ell$.

Now suppose $|C_2| > 1$.
By induction the result holds for $C_2$.  Let $D_1'$ and $D_2'$ be the chord diagrams corresponding to the left and right subtrees of $T_2$.  Now consider how $C_1$ is inserted into $C_2$.  There are three cases
\begin{enumerate}
  \item If $v$ is the root of $T$, then $C_1=D_1$ and $|D_1|=1$, so the lemma is true.
  \item If $v$ is in $D_1$ then $b(D_1) = b(D_1')+1$ and $b(C)= b(C_2)+1$ as label $1$ causes additional shifting while $b(D_2) = b(D_2')$ since the right subtree is not affected.  The induction hypothesis says $b(D_1') \geq b(C_2)-b(D_2')$.  Thus
\[
  b(D_1) \geq b(C_2) - b(D_2')+1 = b(C)-1 - b(D_1)+1 = b(C) - b(D_1)
\]
\item If $v$ is in $D_2$ then $b(D_1)= b(D_1')$ since the right subtree is not affected, while $b(D_2)= b(D_2')+1$ and $b(C)= b(C_2)+1$.  Thus
\[
  b(D_1) \geq b(C_2)-b(D_2') = b(C)-b(D_2)
\]  
\end{enumerate}
This completes the proof.
\end{proof}

\begin{lemma}\label{old claim 2}
  Let $C\in\mathcal{RCCD}$ with $|C|\geq 2$.  Let $T=\mathcal{T}(C)$, and let $H_1$ and $H_2$ be the left and right subtrees, respectively, of $T$. Suppose $D_1$ and $D_2$ are chord diagrams with $\mathcal{T}(D_i) = h_i$.  Then
\[
  f_{C} = f_{D_1}f_{b(D_1)+b(D_2)-b(C)}f_{D_2}
\]
\end{lemma}

\begin{proof}
 Let $C=C_1 \stackrel{\curvearrowright}{(0,k)} C_2$ be the root-share decomposition of $C$.  Let $T_i = \mathcal{T}(C_i)$, and let $v$ be the $k$th vertex of $T$.  
Lemma \ref{first rec claim} gives
\begin{equation}\label{easy rec}
f_{C}f_{b(C)-k} = f_{C_1}f_{b(C_1)-1}f_{C_2}f_{b(C_2)-k+1}
\end{equation}

If $|C_2|=1$ then $C_1=D_1$, $C_2=D_2$, $k=1$, $b(C)=2$, and $b(D_2)=1$, so \eqref{easy rec} gives the result.

Now assume $|C_2|>1$.
Let $D_1'$ and $D_2'$ be the chord diagrams corresponding to the left and right subtrees of $T_2$.  By induction the result holds for $C_2$, that is
\begin{equation}\label{claim in ind}
  f_{C_2} = f_{D_1'}f_{b(D_1')+b(D_2')-b(C_2)}f_{D_2'}
\end{equation}
There are again three cases
  \begin{enumerate}
    \item If $v$ is the root of $T$, then $k=1$ so $C_1=D_1$ and $C_2=D_2$.  Also $b(C_2)=b(C)-1$, and so \eqref{easy rec} becomes the statement of the lemma.
    \item If $v$ is in $D_1$ then $D_1$ is $C_1$ inserted into the $(k-1)^{st}$ slot of $D_1'$, and so 
\begin{equation}\label{easy rec in ind}
  f_{D_1}f_{b(D_1)-k+1} = f_{C_1}f_{b(C_1)-1}f_{D_1'}f_{b(D_1')-k+2}
\end{equation}
Next, note that
  $b(C)  = b(C_2)+1$ and $b(D_1)  = b(D_1')+1$
because inserting shifts the labels up by 1 in the tree being inserted into.
Then substituting \eqref{claim in ind} and \eqref{easy rec in ind} into \eqref{easy rec} we get
\begin{align*}
f_{C}f_{b(C)-k} & = \frac{f_{D_1}f_{b(D_1)-k+1}}{f_{D_1'}f_{b(D_1')-k+2}}f_{D_1'}f_{b(D_1')+b(D_2')-b(C_2)}f_{D_2'}f_{b(C_2)-k+1} \\
& = f_{D_1}f_{b(D_1)+b(D_2)-b(C)}f_{D_2}f_{b(C)-k}
\end{align*}
which implies the statement of the lemma.
\item If $v$ is in $D_2$ then $D_2$ is $C_1$ inserted into the $(k-1-|D_1|)^{st}$ slot of $D_2'$, and so 
\[
  f_{D_2}f_{b(D_2)-k+1+|D_1|} = f_{C_1}f_{b(C_1)-1}f_{D_2'}f_{b(D_2')-k+|D_1|+2}
\]
Similarly to the previous case
  $b(C)  = b(C_2)+1$ and $b(D_2)  = b(D_2')+1$.  So substituting
\begin{align*}
f_{C}f_{b(C)-k} & = \frac{f_{D_2}f_{b(D_2)-k+1+|D_1|}}{f_{D_2'}f_{b(D_2')-k+|D_1|+2}}f_{D_1}f_{b(D_1)+b(D_2')-b(C_2)}f_{D_2'}f_{b(C_2)-k+1} \\
& = f_{D_2}f_{D_1}f_{b(D_1)+b(D_2)-b(C)}f_{b(C)-k}
\end{align*}
which implies the statement of the lemma.
\end{enumerate}
All together, by induction, the result is proved.
\end{proof}

\begin{lemma}\label{old claim 3}
Let $D_1,D_2 \in \mathcal{RCCD}$ and let $G_i = \mathcal{T}(D_i)$.
For every choice of $j$ with $j \geq b(D_2)-1$ and $b(D_1) \geq j-b(D_2)+1$ and for every shuffle of the first $j$ labels of $G_1$ and the first $b(D_2)-1$ labels of $G_2$, there is a unique $T$ in $\mathcal{T}(\mathcal{RCCD})$ with $G_1$ as left child and $G_2$ as right child. 
\end{lemma}

\begin{proof}
Proceed again by induction.
Suppose $G_1$ and $G_2$ both have just one vertex.  So $T$ is the tree with a root with two leaves.  In view of P1 there is exactly one way to label this so that it is in the image of $\mathcal{T}$.  Furthermore, $b(D_2)-1 = 1-1=0$, so the shuffle is trivial.

Take $|g_1|+|g_2| > 2$.
Build a tree $T$ with $G_1$ as right child, $G_2$ as left child, but only label the vertices which had the first $j$ labels in $G_1$ and the first $b(D_2)-1$ labels in $G_2$, and label these leaves by the specified shuffle.

With this labelling there is a leaf with label $1$.  Say $1$ is on the $G_i$ side ($i=1$ or $i=2$).  Let $G_i'$ be $G_i$ with the smallest removable subtree of $G_i$ containing $1$ removed.  Let $T'$ be $T$ built with $G_i'$ in place of $G_i$.  By induction we have a unique labelling of $T'$ consistent with the given shuffle.  

Shift all the labels in $T'$ up by one, and reinsert the smallest removable subtree containing $1$ with all labels other than $1$ shifted to be larger than all labels in $T'$.  This gives a labelling of $T$ which is in the image $\mathcal{T}$.  

To show uniqueness, suppose there were another labelling of $T$ consistent with the given shuffle.  By removing largest removable subtrees containing $1$ as long as they match, we may assume that the two labellings of $T$ have different largest removable subtrees containing $1$.  However removability of a subtree is not sensitive to whether the labels are consecutive, but the two labellings of $T$ only differ by how they shuffle the labels of $G_1$ and $G_2$, so this is impossible.
\end{proof}

Now we are ready to prove Proposition \ref{second rec}

\begin{proof}[Proof of Proposition \ref{second rec}]
 Take a connected rooted chord diagram $C$ with $|C|\geq 2$, and hence $b(C) \geq 2$.  Let the associated tree be $T=\mathcal{T}(C)$.  Let $H_1$ and $H_2$ be the left and right subtrees respectively of $T$.  $H_1$ and $H_2$ satisfy P1 and satisfy P2 recursively since $T$ did; thus $H_1$ and $H_2$ are in the image of $\mathcal{T}$. Let $D_1$ and $D_2$ be the chord diagrams corresponding to $H_1$ and $H_2$.

Let $b(C) = j+1$ and let $b(D_2)= j-\ell+1$, so $\ell=b(C)-b(D_2)$.  $b(C)$ is the label of the leaf of the fully right branch of $T$ and $b(D_2)$ is the label of the leaf of the fully right branch of $H_2$, but these are the same leaf, the only difference is the label shifting.   Thus there must be $\ell$ labels on the $H_1$ side of $T$ which are less than $b(C)$. 

By Lemma \ref{old claim 1} $D_1$ and $D_2$ decompose $C$ into pieces which satisfy the correct conditions to be from a term on the right hand side of \eqref{key lemma}.  Furthermore by Lemma \ref{old claim 2}, they contribute the correct monomial in the $f_i$ to each side.

Now we need to argue the other way.   Suppose $D_1,D_2\in \mathcal{RCCD}$. Viewing $D_1$ and $D_2$ as contributing to a term on the left hand side of the lemma, we need to see that they can be reattached to get a chord diagram $C$, and that this can be done $\binom{j}{\ell}$ ways. 

Let $G_1=\mathcal{T}(D_1)$ and $G_2=\mathcal{T}(D_2)$.  Clearly, we wish to build a tree with $G_1$ as the right child of the root and $G_2$ as the left child.  The question is how to shuffle the leaf labels of $G_1$ and $G_2$ in building the new tree. Lemma \ref{old claim 3} tells us that the answer is $\binom{j}{\ell}$ times, as desired.
\end{proof}

\subsection{The main theorem}

\begin{thm}
\[
\gamma_i(x) = \frac{(-1)^i}{i!}\sum_{\substack{C\in \mathcal{RCCD} \\ b(C) \geq i}}x^{|C|}f_{C}
f_{b(C)-i} 
\]
solves the Dyson-Schwinger equation
\[
G(x,L) = 1 - xG(x,\partial_{-\rho})^{-1}(e^{-L\rho}-1)F(\rho) \big|_{\rho=0}
\]
where
\begin{align*}
F(\rho) & = \frac{f_{0}}{\rho} + f_1 + f_2\rho + f_3\rho^2 + \cdots \\
  G(x,L) & = 1 - \sum_{n \geq 1} \gamma_n(x)L^n
\end{align*}
\end{thm}

\begin{proof}
  By \cite{kythesis} chapter 4, the $\gamma_i$ which solve this Dyson-Schwinger equation satisfy 
\[
\gamma_k(x) = \frac{1}{k}\gamma_1(x)\left(-1+2x\frac{d}{dx}\right)\gamma_{k-1}(x
) \quad \text{for $k\geq 2$}.
\]
This $\gamma_k$ recurrence is a rephrasing of the renormalization group equation for the Dyson-Schwinger equation.

By Proposition \ref{gamma k in g} (see in particular the formulation in \eqref{gamma k rec rephrased}) the $\gamma_i$ defined in the statement of the theorem satisfy the same recurrence.  Thus it suffices to show that $\gamma_1$ as defined in the statement of the theorem satisfies the Dyson-Schwinger equation; namely, we only need to check that  
\begin{equation}\label{interior reduction}
  \gamma_1 = x \left(1 - \sum_{k \geq 1}\gamma_k \frac{d^k}{d(-\rho)^k}\right)^{-1} (-\rho)F(\rho) \big|_{\rho=0}
\end{equation}

To simplify signs let
\[
  g_i = (-1)^ii!\gamma_i
\]
This agrees with the definition of $g_i$ from Proposition \ref{gamma k in g}.
Rephrasing in terms of the $g_k$ and expanding the geometric series, \eqref{interior reduction} is equivalent to checking
\[
  g_1 = x\sum_{n\geq 0}\left(\sum_{\ell \geq 1}g_\ell \frac{1}{\ell!}\frac{d^\ell}{d\rho^\ell}\right)^n (f_{0} + f_1\rho+f_2\rho^2+\cdots)\big|_{\rho=0}
\]
Simplifying the right hand side we get
\begin{align*}
&x\sum_{n\geq 0}\left(\sum_{\ell \geq 1}g_\ell \frac{1}{\ell!}\frac{d^\ell}{d\rho^\ell}\right)^n (f_{0} + f_1\rho+f_2\rho^2+\cdots)\big|_{\rho=0} \\
& = xf_{0} + x\sum_{n\geq 1}\left(\sum_{\ell \geq 1}g_\ell \frac{1}{\ell!}\frac{d^\ell}{d\rho^\ell}\right)^n (f_{0} + f_1\rho+f_2\rho^2+\cdots)\big|_{\rho=0} \\
\end{align*}

So it suffices to show
\[
g_1= xf_{0} + x\sum_{n\geq 1}\left(\sum_{\ell \geq 1}g_\ell \frac{1}{\ell!}\frac{d^\ell}{d\rho^\ell}\right)^n (f_{0} + f_1\rho+f_2\rho^2+\cdots)\big|_{\rho=0}
\]
This is the content of the next lemma, following which the proof is complete.
\end{proof}

\begin{lemma}
Let
\[
  g_i = \sum_{\substack{C\in \mathcal{RCCD}\\b(C)\geq i}}x^{|C|}f_{C}f_{b(C)-i}
\]
then
\begin{equation}\label{main thm eqn}
g_1= xf_{0} + x\sum_{n\geq 1}\left(\sum_{\ell \geq 1}g_\ell \frac{1}{\ell!}\frac{d^\ell}{d\rho^\ell}\right)^n (f_{0} + f_1\rho+f_2\rho^2+\cdots)\big|_{\rho=0}
\end{equation}
\end{lemma}

\begin{proof}
  One can check directly that the linear term of $g_1$ is correct.
  
  For $i\geq 1$ we have
  \[
  [x^{i+1}]g_1 = \sum_{\substack{C\in\mathcal{RCCD}\\b(C)\geq 1\\|C|=i+1}}f_{C}f_{b(C)-1}
  \]
  while the coefficient of $x^{i+1}$ on the right hand side of \eqref{main thm eqn} is
  \[
  [x^{i}]\sum_{n\geq 1}\left(\sum_{\ell \geq 1}g_\ell \frac{1}{\ell!}\frac{d^\ell}{d\rho^\ell}\right)^n (f_{0} + f_1\rho+f_2\rho^2+\cdots)\big|_{\rho=0} \\
  \] 
  One possible way for these to be equal is if the $f_j$ explicitly showing on the right hand side matches with the $f_{b(C)-1}$ from the left hand side.  That is it would suffice to show for $1 \leq j \leq i$
  \begin{equation}\label{deriving lemma}
    \sum_{\substack{C\in\mathcal{RCCD}\\b(C) = j+1\\|C|=i+1}}f_{C} = [x^{i}]\sum_{n\geq 1}\left(\sum_{\ell \geq 1}g_\ell \frac{1}{\ell!}\frac{d^\ell}{d\rho^\ell}\right)^n \rho^{j}\big|_{\rho=0}
  \end{equation}
  To save space, let $G_\rho = \sum_{\ell \geq 1}g_\ell \frac{1}{\ell!}\frac{d^\ell}{d\rho^\ell}$, and let $F_{i,j} = [x^{i}]\sum_{n\geq 0}\left(\sum_{\ell \geq 1}g_\ell \frac{1}{\ell!}\frac{d^\ell}{d\rho^\ell}\right)^n \rho^{j}\big|_{\rho=0}$.  Note that for $j\geq 1$, the $n=0$ term does not contribute and so $F_{i,j}$ is the right hand side of \eqref{deriving lemma}.  Calculate, for $1 \leq j \leq i$
  \begin{align*}
    F_{i,j} = & [x^{i}]\sum_{n\geq 1}\left(G_\rho\right)^n \rho^{j}\big|_{\rho=0} \\
    & = \sum_{k=1}^{i}\left([x^k]G_\rho\right)\left([x^{i-k}]\sum_{n\geq 0}\left(G_\rho\right)^n\right)\rho^j\big|_{\rho=0} \\
    & = \sum_{k=1}^{i}\sum_{\ell=1}^{j}\binom{j}{\ell}\left([x^k]G_d \rho^{\ell}\big|_{\rho=0}\right)\left([x^{i-k}]\sum_{n\geq 0}\left(G_\rho\right)^n \rho^{j-\ell}\big|_{\rho=0}\right) \\
    & = \sum_{k=1}^{i}\sum_{\ell=1}^{j}\binom{j}{\ell}\left([x^k]g_\ell\right)F_{i-k,j-\ell} \\
    & =  \sum_{k=1}^{i}\sum_{\ell=1}^{j}\binom{j}{\ell}\left(\sum_{\substack{C\in\mathcal{RCCD}\\b(C)\geq \ell \\|C|=k}}f_{i(C)}f_{b(C)-\ell}\right)F_{i-k,j-\ell}
  \end{align*}

Thus we have a recurrence which gives $F_{i,j}$.  To prove the lemma it suffices to prove \eqref{deriving lemma}; that is, it suffices to prove that
\[
  F_{i,j} = \sum_{\substack{C\in\mathcal{RCCD}\\b(C) = j+1\\|C|=i+1}}f_{C}
\]
for $1\leq i \leq j$.  It still suffices to prove it with extended bounds $0\leq i \leq j$.
To do this we check the initial terms directly and then check that 
$
\sum_{\substack{C\in\mathcal{RCCD}\\b(C) = j+1\\|C|=i+1}}f_{C}
$
satisfies the recurrence for $F_{i,j}$.  The second of these is the content of Proposition \ref{second rec}.  For the first of these, if $j=0$ then 
\begin{align*}
F_{i,0} & = [x^{i}]\sum_{n\geq 0}\left(\sum_{\ell \geq 1}g_\ell \frac{1}{\ell!}\frac{d^\ell}{d\rho^\ell}\right)^n \rho^0\big|_{\rho=0} \\
& = [x^i]1 = \begin{cases} 1 & \text{if $i=0$} \\ 0 & \text{otherwise}\end{cases}
\end{align*}
while
\[
\sum_{\substack{C\in\mathcal{RCCD}\\b(C) = 0+1\\|C|=i+1}}f_{C} = \begin{cases} 1 & \text{if $i=0$} \\ 0 & \text{otherwise}\end{cases}
\]
as the only rooted connected cord diagram with $b(C)=1$ is the diagram with exactly one chord.
This completes the proof.
\end{proof}

\section{Consequences and conclusions}\label{conclusion section}

\subsection{Asymptotic analysis}

Given the Laurent series expansion $F(\rho)$ we obtained a formal solution $G(x,L)$ of the analytic Dyson--Schwinger equation \eqref{DSE}:

\begin{equation*}
\left \{
\begin{aligned}
&G(x,L)\,=\,1 - \sum_{k\in \mathbb{N}^*} \gamma_{k} (x) \, L^k \,,\\
& \gamma_{k}(x)\in \mathbb{C}[[x]]\,
\end{aligned}
\right .
\end{equation*}

\noindent with an explicit expression for the series $\gamma_{k}(x)$ in terms of the coefficients of $F(\rho)$,

\begin{equation*}
\gamma_{k}(x) \,=\, \sum_{\substack{ X \in \mathcal{RCCD} \\  b(x) \geq k}} f_{X}\,f_{b(X)-k} \,x^{|X|}\,.
\end{equation*}

So by choosing the analytic properties of $F(\rho)$ we are on good grounds to study the analytic properties of the $\gamma_{k}$. This will be done in a future work.\\

As a preview lets look at an example of Gevrey classification of the $\gamma_{k}$. Remember that we say $\gamma_{k}(x)$ is a series of Gevrey class $q\in \mathbb{R}_{+}$ if there are positive constant $K_{k}$ and $A_{k}$ such that for all $n\in \mathbb{N}$ its coefficients satisfy

\begin{equation*}
|\gamma_{k,n}| \, \leq\, A_{k} \, K_{k}^n  \, (n!)^q\,.
\end{equation*}

\noindent We have the following result.

\begin{prop}
 Assume that there is a positive constant $C$ such that for all $n\in \mathbb{N}$ the coefficients of $F(\rho)$ satisfy $|f_{n}| \leq C^{n+1}$, then for all $k \in \mathbb{N}^*$ the formal power series $\gamma_{k}(x)$ is of Gevrey class $1$.
\end{prop}
\begin{proof}
We need to show that for all $k \in \mathbb{N}^*$ and $n \in \mathbb{N}$
\begin{equation*}
\left | \sum_{|X|=n\,,\, b(X)\geq k} \,f_{X}\,f_{b(X)-k} \right | \leq A_{k}\,B_{k}^{n}\,n!
\end{equation*}
with $A_{k},B_{k} \in \mathbb{R}_{+}$. To do so we write the monomial $f_{X}=f_{0}^{p_{0}(X)} \cdots f_{n}^{p_{n}(X)} $ with $p_{i}(X)$ the number of times the factor $f_{i}$ appears in the product $f_{X}$.\\

\noindent We have $p_{0}(X)+\cdots+p_{n}(X)=n-1$, the length of $\bar{\delta}(X)$. So we get:
\begin{align*}
\left | \sum_{|X|=n\,,\,b(X)\geq k} f_{0}^{p_{0}(X)} \cdots f_{n}^{p_{n}(X)} \, f_{b(X)-k} \right | & \leq \sum_{|X|=n\,,\,b(X)\geq k} |f_{0}|^{p_{0}(X)} \cdots |f_{n}|^{p_{n}(X)} \, |f_{b(X)-k}|\\
& \leq \sum_{|X|=n\,,\,b(X)\geq k} C^{p_{0}(X)+2\,p_{1}(X)+ \cdots + (n+1)\,p_{n}(X)}\,C^{b(X)-k+1}\,.
\end{align*} 

\noindent But since $p_{0}(X)+\cdots+p_{n}(X)=n-1$, for each diagram $X$ there are some integers $P(X)$ and $Q(X)$ such that
\begin{equation*}
\sum_{i=0}^{n} (i+1)\,p_{i}(X) =P(X) \, n + Q(X). 
\end{equation*}

\noindent Hence with $k \leq b(X) \leq n$ we obtain the bound
\begin{equation*}
\left | \sum_{|X|=n\,,\, b(X)\geq k} \,f_{X}\,f_{b(X)-k} \right | \leq \sum_{|X|=n\,,\,b(X)\geq k} C^{P(X) n + Q(X) + b(X) -k + 1} \leq \sum_{|X|=n\,,\,b(X)\geq k} C^{P_{k} n + Q_{k}} \, c_{n,k} 
\end{equation*}

\noindent with $P_{k},Q_{k}$ some integers and $c_{n,k}$ the number of rooted connected chord diagrams of degree $n$ with smallest terminal chord larger than $k$. So for all $k \in \mathbb{N}$ this is bounded by the total number of rooted chord diagrams of degree $n$ i.e.
\begin{equation*}
c_{n,k} \leq (2n-1)!! \leq (2n)!! = 2^n \, n!\,.
\end{equation*}

\noindent Putting everything together we get the desired bound:

\begin{equation*}
\left | \sum_{|X|=n\,,\, b(X)\geq k} \,f_{X}\,f_{b(X)-k} \right | \leq C{Q_{k}}\,(2C^{P_{k}})^{n}\,n!\,.
\end{equation*}

\end{proof}

For the $\gamma_{k}$ satisfying the conditions of this proposition, applying the Borel transform gives series with a non zero radius of convergence so that they are amenable to a study of their Borel summability properties (see \cite{Bal}).

\subsection{$P$ and the differential equations}

In \cite{kythesis} more general Dyson-Schwinger equations, including the one case considered in the present paper, are converted into differential equations which are then analysed in some particular cases in \cite{vBKUY,vBKUY2}.  The principle difficulty of this method is that this conversion process builds a new series $P(x)$ out of the primitive graphs.  $P$ is not generally well understood and so all results must be conditional on assumptions on $P$.  

The chord diagram techniques considered above give a solution to the one particular Dyson-Schwinger equation \eqref{DSE}, and so in that case give
\[
  P(x) = \sum_{\substack{C\in \mathcal{RCCD} \\ b(C) \geq 1}}x^{|C|}f_{C}(f_{b(C)-2}-f_{b(C)-1})
\]
This shows us that in this case $P(x)$ has an expansion which is a modified form of the expansion $\gamma_1(x)$, with the final factor of $f_k$ replaced by $f_{k-1}-f_k$.  We expect a similar overall shape to hold for other Dyson-Schwinger equations which gives us some basis on which to evaluate the reasonableness of assumptions on $P$.

\subsection{Four-term relation}

We are tempted to try to make a connection with the classical subject of the study of algebraic structures on chord diagrams (and their linear / rooted version) linked to the theory of Vassiliev invariants of knots \cite{cdm}. Indeed non connected chord diagrams do not appear in the $\gamma_{k}$ expansions, which we can interpret as saying that the monomials already satisfy a one-term relation. Hence it seems legitimate to ask whether or not these monomials define weight systems on rooted chord diagrams. So do these products of $f_k$ satisfy a four-term relation? \\

Here we give a minimal (for the number of chords involved in the diagrams) counter-example proving that it is not the case in general, that is when we consider the coefficients of $F(\rho)$ as formal objects. Define a family of maps $M_{\alpha}: \mathcal{RCCD} \longrightarrow \mathbb{C}[f_{0},f_{1},...] $, $X \mapsto  M_{\alpha}(X)=f_{X}\,f_{b(X)-\alpha(X)}$ which associates to $X$ a monomial in the $f_{k}$ where $\alpha: \mathcal{RCCD} \longrightarrow \ZZ$ plays the role of a parameter. Then the expansion $\gamma_{k}(x)$ corresponds to the choice of $\alpha$ being the constant map equal to $k$. We are going to show that, for $\alpha$ satisfying some reasonable properties, the map $M_{\alpha}$ does not satisfy a four-term relation. But lets make these notions precise.

\begin{definition}
A sequence $(A,B,C,D)$ of rooted connected chord diagrams is said to be in a \textbf{four-term configuration} if they differ precisely in the following way:

\begin{figure}[H]
\begin{center}
\hspace{-1cm}
\begin{pspicture}(0,0)(12.8,3.2)

\psarc[linecolor=red,linestyle=solid]{-}(1.25,1.9){1.15}{0}{30}
\psarc[linecolor=red,linestyle=solid]{-}(1.25,1.9){1.15}{150}{190}
\psarc[linecolor=red,linestyle=solid]{-}(1.25,1.9){1.15}{245}{300}
\psarc[linecolor=black,linestyle=dashed]{-}(1.25,1.9){1.15}{38}{150}
\psarc[linecolor=black,linestyle=dashed]{-}(1.25,1.9){1.15}{190}{245}
\psarc[linecolor=black,linestyle=dashed]{-}(1.25,1.9){1.15}{300}{360}
\psrotate(1.25,1.9){150}{\psarc[linecolor=red]{*-*}(2,4.75){2.5}{233}{278}}
\psrotate(1.25,1.9){245}{\psarc[linecolor=red]{*-*}(2,4.75){2.5}{233}{278}}
\rput(1.25,0.40){\small$(A)$}

\rput(3.1,1.900){,}

\psrotate(5,1.9){135}{\psarc[linecolor=red]{*-*}(5.75,4.75){2.5}{233}{278}}
\psrotate(5,1.9){260}{\psarc[linecolor=red]{*-*}(5.75,4.75){2.5}{233}{278}}
\psarc[linecolor=red,linestyle=solid]{-}(5,1.9){1.15}{10}{50}
\psarc[linecolor=red,linestyle=solid]{-}(5,1.9){1.15}{135}{180}
\psarc[linecolor=red,linestyle=solid]{-}(5,1.9){1.15}{245}{300}
\psarc[linecolor=black,linestyle=dashed]{-}(5,1.9){1.15}{54}{130}
\psarc[linecolor=black,linestyle=dashed]{-}(5,1.9){1.15}{185}{245}
\psarc[linecolor=black,linestyle=dashed]{-}(5,1.9){1.15}{300}{360}
\rput(5,0.40){\small$(B)$}

\rput(6.85,1.900){,}

\psarc[linecolor=red,linestyle=solid]{-}(8.75,1.9){1.15}{10}{50}
\psarc[linecolor=red,linestyle=solid]{-}(8.75,1.9){1.15}{135}{180}
\psarc[linecolor=red,linestyle=solid]{-}(8.75,1.9){1.15}{245}{300}
\psarc[linecolor=black,linestyle=dashed]{-}(8.75,1.9){1.15}{54}{130}
\psarc[linecolor=black,linestyle=dashed]{-}(8.75,1.9){1.15}{185}{245}
\psarc[linecolor=black,linestyle=dashed]{-}(8.75,1.9){1.15}{300}{360}
\psrotate(8.75,1.9){25}{\psarc[linecolor=red]{*-*}(9.5,4.75){2.5}{233}{278}}
\psrotate(8.75,1.9){255}{\psarc[linecolor=red]{*-*}(9.5,4.75){2.5}{233}{278}}
\rput(8.75,0.40){\small$(C)$}

\rput(10.65,1.900){,}

\psrotate(12.5,1.9){5}{
\psarc[linecolor=red,linestyle=solid]{-}(12.5,1.9){1.15}{10}{50}
\psarc[linecolor=red,linestyle=solid]{-}(12.5,1.9){1.15}{135}{180}
\psarc[linecolor=red,linestyle=solid]{-}(12.5,1.9){1.15}{245}{300}
\psarc[linecolor=black,linestyle=dashed]{-}(12.5,1.9){1.15}{54}{130}
\psarc[linecolor=black,linestyle=dashed]{-}(12.5,1.9){1.15}{185}{245}
\psarc[linecolor=black,linestyle=dashed]{-}(12.5,1.9){1.15}{300}{360}
}
\psrotate(12.5,1.9){15}{\psarc[linecolor=red]{*-*}(13.25,4.75){2.5}{233}{278}}
\psrotate(12.5,1.9){275}{\psarc[linecolor=red]{*-*}(13.25,4.75){2.5}{233}{278}}
\rput(12.5,0.4){\small$(D)$}

\rput(14.5,1.900){.}

\end{pspicture}
\end{center}
\end{figure}
\end{definition}

\noindent Now if $(A,B,C,D)$ is in a four-term configuration we want to evaluate:
\begin{equation*}
\left<(A,B,C,D),M_{\alpha} \right>_{4T}= M_{\alpha}(A)-M_{\alpha}(B)+M_{\alpha}(C)-M_{\alpha}(D).
\end{equation*}

\begin{definition}
We say that $M_{\alpha}$ satisfies a four-term relation if for all the sequences $(A,B,C,D)$ in four-term configurations we have $\left<(A,B,C,D),M_{\alpha} \right>_{4T}=0$.
\end{definition}

\noindent We show that for a reasonable $\alpha$ the map $M_{\alpha}$ does not satisfy a four-term relation.

\begin{prop}
If the map $M_{\alpha}$ satisfies a four-term relation then $\alpha: \mathcal{RCCD} \longrightarrow \mathbb{Z}$ is a multivalued map.
\end{prop}
\begin{proof}
We exhibit a chord diagram $X$ with minimal degree such that $\alpha(X)$ depends on the four-term configuration where we find it. So consider the following four-term configurations:

\begin{figure}[H]
\begin{center}
\hspace{-1cm}
\begin{pspicture}(0,0.5)(12.8,3.4)

\pscircle[linestyle=solid](1.25,1.9){1.5}
\psarc[linecolor=red,linestyle=solid]{*-*}(-1.45,-0.3){3}{14}{64.5}
\psrotate(1.25,1.9){100}{\psarc[linecolor=red,linestyle=solid]{*-*}(-1.45,-0.3){3}{14}{64.5}}
\psrotate(1.25,1.9){-65}{\psarc[linecolor=black,linestyle=solid]{*-*}(-1.45,-0.3){3}{14}{64.5}}
\psarc[linecolor=black,linestyle=solid]{*-*}(1.25,-4){6}{75.5}{104.5}
\pscircle(1.12,3.38){0.12}
\rput(1.25,0.10){\small$(A)$}

\rput(3.1,1.900){,}

\pscircle[linestyle=solid](5,1.9){1.5}
\psrotate(5,1.9){-15}{\psarc[linecolor=red,linestyle=solid]{*-*}(2.3,-0.3){3}{14}{64.5}}
\psrotate(5,1.9){120}{\psarc[linecolor=red,linestyle=solid]{*-*}(2.3,-0.3){3}{14}{64.5}}
\psrotate(5,1.9){-65}{\psarc[linecolor=black,linestyle=solid]{*-*}(2.3,-0.3){3}{14}{64.5}}
\psarc[linecolor=black,linestyle=solid]{*-*}(5,-4){6}{75.5}{104.5}
\pscircle(4.87,3.38){0.12}
\rput(5,0.10){\small$(B)$}

\rput(6.85,1.900){,}

\pscircle[linestyle=solid](8.75,1.9){1.5}
\psrotate(8.75,1.9){100}{\psarc[linecolor=red,linestyle=solid]{*-*}(6.05,-0.3){3}{14}{64.5}}
\psrotate(8.75,1.9){240}{\psarc[linecolor=red,linestyle=solid]{*-*}(6.05,-0.3){3}{14}{64.5}}
\psrotate(8.75,1.9){-65}{\psarc[linecolor=black,linestyle=solid]{*-*}(6.05,-0.3){3}{14}{64.5}}
\psarc[linecolor=black,linestyle=solid]{*-*}(8.75,-4){6}{75.5}{104.5}
\pscircle(8.62,3.38){0.12}
\rput(8.75,0.10){\small$(X)$}

\rput(10.65,1.900){,}

\pscircle[linestyle=solid](12.5,1.9){1.5}
\psrotate(12.5,1.9){125}{\psarc[linecolor=red,linestyle=solid]{*-*}(9.8,-0.3){3}{14}{64.5}}
\psrotate(12.5,1.9){225}{\psarc[linecolor=red,linestyle=solid]{*-*}(9.8,-0.3){3}{14}{64.5}}
\psrotate(12.5,1.9){-65}{\psarc[linecolor=black,linestyle=solid]{*-*}(9.8,-0.3){3}{14}{64.5}}
\psarc[linecolor=black,linestyle=solid]{*-*}(12.5,-4){6}{75.5}{104.5}
\pscircle(12.37,3.38){0.12}
\rput(12.5,0.1){\small$(D)$}


\end{pspicture}
\end{center}
\end{figure}

\noindent which gives the sum

\begin{equation*}
\left<(A,B,X,D),M_{\alpha} \right>_{4T}=f_{0}^2 \left( f_{0}f_{4-\alpha(A)} - f_{0}f_{4-\alpha(B)} + f_{2}f_{2-\alpha(X)} - f_{0}f_{4-\alpha(D)} \right)
\end{equation*}

\noindent and

\begin{figure}[H]
\begin{center}
\hspace{-1cm}
\begin{pspicture}(0,0.5)(12.8,3.4)

\pscircle[linestyle=solid](1.25,1.9){1.5}
\psarc[linecolor=red,linestyle=solid]{*-*}(-1.45,-0.3){3}{14}{64.5}
\psrotate(1.25,1.9){100}{\psarc[linecolor=red,linestyle=solid]{*-*}(-1.45,-0.3){3}{14}{64.5}}
\psrotate(1.25,1.9){-180}{\psarc[linecolor=black,linestyle=solid]{*-*}(-1.45,-0.3){3}{14}{64.5}}
\psarc[linecolor=black,linestyle=solid]{*-*}(1.25,-4){5}{77.5}{102.5}
\pscircle(0.17,0.88){0.12}
\rput(1.25,0.10){\small$(A')$}

\rput(3.1,1.900){,}

\pscircle[linestyle=solid](5,1.9){1.5}
\psrotate(5,1.9){-15}{\psarc[linecolor=red,linestyle=solid]{*-*}(2.3,-0.3){3}{14}{64.5}}
\psrotate(5,1.9){120}{\psarc[linecolor=red,linestyle=solid]{*-*}(2.3,-0.3){3}{14}{64.5}}
\psrotate(5,1.9){-180}{\psarc[linecolor=black,linestyle=solid]{*-*}(2.3,-0.3){3}{14}{64.5}}
\psarc[linecolor=black,linestyle=solid]{*-*}(5,-4){5}{77.5}{102.5}
\pscircle(3.92,0.88){0.12}
\rput(5,0.10){\small$(X)$}

\rput(6.85,1.900){,}

\pscircle[linestyle=solid](8.75,1.9){1.5}
\psrotate(8.75,1.9){100}{\psarc[linecolor=red,linestyle=solid]{*-*}(6.05,-0.3){3}{14}{64.5}}
\psrotate(8.75,1.9){240}{\psarc[linecolor=red,linestyle=solid]{*-*}(6.05,-0.3){3}{14}{64.5}}
\psrotate(8.75,1.9){-180}{\psarc[linecolor=black,linestyle=solid]{*-*}(6.05,-0.3){3}{14}{64.5}}
\psarc[linecolor=black,linestyle=solid]{*-*}(8.75,-4){5}{77.5}{102.5}
\pscircle(7.67,0.88){0.12}
\rput(8.75,0.10){\small$(C')$}

\rput(10.65,1.900){,}

\pscircle[linestyle=solid](12.5,1.9){1.5}
\psrotate(12.5,1.9){125}{\psarc[linecolor=red,linestyle=solid]{*-*}(9.8,-0.3){3}{14}{64.5}}
\psrotate(12.5,1.9){225}{\psarc[linecolor=red,linestyle=solid]{*-*}(9.8,-0.3){3}{14}{64.5}}
\psrotate(12.5,1.9){-180}{\psarc[linecolor=black,linestyle=solid]{*-*}(9.8,-0.3){3}{14}{64.5}}
\psarc[linecolor=black,linestyle=solid]{*-*}(12.5,-4){5}{77.5}{102.5}
\pscircle(11.42,0.88){0.12}
\rput(12.5,0.1){\small$(D')$}


\end{pspicture}
\end{center}
\end{figure}

\noindent which gives the sum

\begin{equation*}
\left<(A',X,C',D'),M_{\alpha} \right>_{4T}=f_{0}^2 \left( f_{1}f_{3-\alpha(A')} - f_{2}f_{2-\alpha(X)} + f_{0}f_{4-\alpha(C')} - f_{0}f_{4-\alpha(D')} \right)\,.
\end{equation*}

\noindent So if $M_{\alpha}$ satisfies a four-term relation we must have at the same time $\alpha(X)=2$ and $\alpha(X)=1$ i.e. $\alpha$ is multivalued.

\end{proof}

As a consequence there can be no four-term relation linking the monomials appearing in the series $\gamma_{k}$ since the map $\alpha$ is a constant map in this situation.\\

Of course this does not rule out the possibility of a four-term relation holding for some specific $F(\rho)$ and it would be interesting to exhibit non trivial examples of such a situation. There might also be a four-term relation satisfied by $M_{\alpha}$ if we allow a sum over all possible roots in a four-term configuration. It is finally possible that there are non-trivial general relations among the monomials $M_{\alpha}(X)$ that are not related to the algebra of linear chord diagrams. Investigating these questions would require a more detailed understanding of the distribution of terminal chords in the rooted connected chord diagrams and will be done in a future work.

\appendix

\section{The objects}\label{obj appendix}

The following table contains all the rooted connected chord diagrams up to $4$ chords together with their corresponding rooted planar binary trees. The chords of the diagrams and the leaves of the trees are labelled in the intersection order.


\begin{figure}[H]
 \begin{center}
 \hspace{3cm}
\begin{pspicture}(0,1)(12.8,1)
\rput(5,1){
\psshadowbox{\bf Chord Diagrams and their Planar Binary Trees ; $n=1$ }
}
\end{pspicture}
\end{center}
\end{figure} \hspace{-0.5cm}

\begin{figure}[H]

 \begin{center}
 \hspace{3cm}
\begin{pspicture}(0,0)(12.8,3.2)
\pscircle(2.5,1.900){1.25}
\psline{*-*}(2.5,0.65)(2.5,3.15)
\pscircle[linecolor=red](2.5,0.65){0.12}
\rput(2.5,0.4){\tiny$1$}

\rput(7,2.1){
\pstree[treesep=0.5cm,levelsep=0.7cm]{\Tc*{0pt}}{\pstree[treesep=1.2cm,levelsep=0.5cm]{\Tc*{1.2pt}}{}}
}
\pscircle[linecolor=black](7,2){0.12}
\rput(7,1.7){\tiny$1$}

\end{pspicture}
\end{center}

\end{figure} \hspace{0.5cm}


\begin{figure}[H]
 \begin{center}
 \hspace{3cm}
\begin{pspicture}(0,1)(12.8,1)
\rput(5,1){
\psshadowbox{\bf Chord Diagrams and their Planar Binary Trees ; $n=2$ }
}
\end{pspicture}
\end{center}
\end{figure} \hspace{-0.5cm}

\begin{figure}[H]

 \begin{center}
 \hspace{3cm}
\begin{pspicture}(0,0)(12.8,3.2)
\pscircle(2.5,1.900){1.25}
\psline{*-*}(2.5,0.65)(2.5,3.15)
\psarc[linecolor=black]{*-*}(2.5,3.15){1.3}{212}{328}
\pscircle[linecolor=red](2.5,0.65){0.12}
\rput(2.5,0.4){\tiny$1$}
\rput(3.75,2.6){\tiny$2$}

\rput(7,2){
\pstree[treesep=0.5cm,levelsep=0.5cm]{\Tc*{0pt}}{\pstree[treesep=1.2cm,levelsep=0.7cm]{\Tc*{1.2pt}}{\Tc*{1.2pt} \Tc*{1.2pt}}}
}
\pscircle(7,2.1){0.12}
\rput(6.35,1.2){\tiny$1$}
\rput(7.65,1.2){\tiny$2$}

\end{pspicture}
\end{center}

\end{figure} \hspace{0.5cm}


\begin{figure}[H]
 \begin{center}
 \hspace{3cm}
\begin{pspicture}(0,1)(12.8,1)
\rput(5,1){
\psshadowbox{\bf Chord Diagrams and their Planar Binary Trees ; $n=3$ }
}
\end{pspicture}
\end{center}
\end{figure} \hspace{-0.5cm}

\begin{figure}[H]

 \begin{center}
 \hspace{3cm}
\begin{pspicture}(0,0)(12.8,3.2)
\pscircle(-2,1.900){1.25}
\psline{*-*}(-2,0.65)(-2,3.15)
\psarc[linecolor=black]{*-*}(-2,3.15){1}{204}{336}
\psrotate(-2,1.9){180}{\psarc[linecolor=black]{*-*}(-2,3.15){1}{204}{336}}
\pscircle[linecolor=red](-2,0.65){0.12}
\rput(-2,0.4){\tiny$1$}
\rput(-0.9,0.9){\tiny$2$}
\rput(-0.9,2.9){\tiny$3$}

\rput(2.2,2){
\pstree[treesep=0.5cm,levelsep=0.5cm]{\Tc*{0pt}}{\pstree[treesep=1.2cm,levelsep=0.7cm]{\Tc*{1.2pt}}{\pstree[treesep=1.2cm,levelsep=0.7cm]{\Tc*{1.2pt}}{\Tc*{1.2pt} \Tc*{1.2pt}} \Tc*{1.2pt}}}
}
\pscircle(2.525,2.45){0.12}
\rput(1.25,0.9){\tiny$1$}
\rput(2.55,0.9){\tiny$3$}
\rput(3.175,1.6){\tiny$2$}

\pscircle(7,1.900){1.25}
\psarc[linecolor=black]{*-*}(7.3,0.5){1}{44}{160}
\psrotate(7,1.9){60}{\psarc[linecolor=black]{*-*}(7.3,0.5){1}{44}{160}}
\psrotate(7,1.9){120}{\psarc[linecolor=black]{*-*}(7.3,0.5){1}{44}{160}}
\pscircle[linecolor=red](6.35,0.835){0.12}
\rput(6.35,0.6){\tiny$1$}
\rput(7.65,0.6){\tiny$2$}
\rput(8.45,1.85){\tiny$3$}

\rput(11.8,2){
\pstree[treesep=0.5cm,levelsep=0.5cm]{\Tc*{0pt}}{\pstree[treesep=1.2cm,levelsep=0.7cm]{\Tc*{1.2pt}}{\Tc*{1.2pt} \pstree[treesep=1.2cm,levelsep=0.7cm]{\Tc*{1.2pt}}{\Tc*{1.2pt} \Tc*{1.2pt}} }}
}
\pscircle[linecolor=black](11.475,2.45){0.12}
\rput(11.5,0.9){\tiny$2$}
\rput(12.765,0.9){\tiny$3$}
\rput(10.85,1.6){\tiny$1$}

\end{pspicture}
\end{center}

\end{figure} \hspace{0.5cm}

\begin{figure}[H]

 \begin{center}
 \hspace{3cm}
\begin{pspicture}(0,0)(12.8,3.2)

\psrotate(-2,1.9){130}{
\pscircle(-2,1.900){1.25}
\psline{*-*}(-2,0.65)(-2,3.15)
\psarc[linecolor=black]{*-*}(-2,3.15){1}{204}{336}
\psrotate(-2,1.9){180}{\psarc[linecolor=black]{*-*}(-2,3.15){1}{204}{336}}
}
\pscircle[linecolor=red](-2.175,0.65){0.12}
\rput(-2.175,0.4){\tiny$1$}
\rput(-0.725,1.75){\tiny$2$}
\rput(-1.05,2.85){\tiny$3$}

\rput(2.2,2){
\pstree[treesep=0.5cm,levelsep=0.5cm]{\Tc*{0pt}}{\pstree[treesep=1.2cm,levelsep=0.7cm]{\Tc*{1.2pt}}{\Tc*{1.2pt} \pstree[treesep=1.2cm,levelsep=0.7cm]{\Tc*{1.2pt}}{\Tc*{1.2pt} \Tc*{1.2pt}} }}
}
\pscircle[linecolor=black](1.885,2.45){0.12}
\rput(1.89,0.9){\tiny$1$}
\rput(3.165,0.9){\tiny$3$}
\rput(1.25,1.6){\tiny$2$}

\pscircle(7,1.900){1.25}
\psline{*-*}(7,0.65)(7,3.15)
\psrotate(7,1.9){60}{\psline{*-*}(7,0.65)(7,3.15)}
\psrotate(7,1.9){120}{\psline{*-*}(7,0.65)(7,3.15)}
\pscircle[linecolor=red](7,0.65){0.12}
\rput(7,0.4){\tiny$1$}
\rput(8.3,1.2){\tiny$2$}
\rput(8.25,2.65){\tiny$3$}

\rput(11.8,2){
\pstree[treesep=0.5cm,levelsep=0.5cm]{\Tc*{0pt}}{\pstree[treesep=1.2cm,levelsep=0.7cm]{\Tc*{1.2pt}}{\pstree[treesep=1.2cm,levelsep=0.7cm]{\Tc*{1.2pt}}{\Tc*{1.2pt} \Tc*{1.2pt}} \Tc*{1.2pt}}}
}
\pscircle(12.125,2.45){0.12}
\rput(10.85,0.9){\tiny$1$}
\rput(12.15,0.9){\tiny$2$}
\rput(12.825,1.6){\tiny$3$}

\end{pspicture}
\end{center}

\end{figure} \hspace{0.5cm}


\begin{figure}[H]
 \begin{center}
 \hspace{3cm}
\begin{pspicture}(0,1)(12.8,1)
\rput(5,1){
\psshadowbox{\bf Chord Diagrams and their Planar Binary Trees ; $n=4$ }
}
\end{pspicture}
\end{center}
\end{figure} \hspace{-0.5cm}

\begin{figure}[H]

 \begin{center}
 \hspace{3cm}
\begin{pspicture}(0,0)(12.8,3.2)

\pscircle(-2,1.900){1.25}
\psrotate(-2,1.9){200}{\psarc[linecolor=black]{*-*}(-2,3.45){1.1}{217}{323}}
\psrotate(-2,1.9){260}{\psarc[linecolor=black]{*-*}(-2,3.45){0.9}{217}{323}}
\psrotate(-2,1.9){320}{\psarc[linecolor=black]{*-*}(-2,3.45){0.9}{217}{323}}
\psrotate(-2,1.9){380}{\psarc[linecolor=black]{*-*}(-2,3.45){0.9}{217}{323}}

\pscircle[linecolor=red](-2.515,0.77){0.12}

\rput(-2.495,0.5){\tiny$1$}
\rput(-0.925,0.95){\tiny$2$}
\rput(-0.63,2.2){\tiny$3$}
\rput(-1.52,3.25){\tiny$4$}

\rput(2.2,2){
\pstree[treesep=0.5cm,levelsep=0.5cm]{\Tc*{0pt}}{\pstree[treesep=1.2cm,levelsep=0.7cm]{\Tc*{1.2pt}}{\Tc*{1.2pt} \pstree[treesep=1.2cm,levelsep=0.7cm]{\Tc*{1.2pt}}{\Tc*{1.2pt} \pstree[treesep=1.2cm,levelsep=0.7cm]{\Tc*{1.2pt}}{\Tc*{1.2pt} \Tc*{1.2pt}}} }}
}
\pscircle[linecolor=black](1.565,2.8){0.12}
\rput(0.93,1.95){\tiny$1$}
\rput(1.55,1.25){\tiny$2$}
\rput(2.2,0.55){\tiny$3$}
\rput(3.5,0.55){\tiny$4$}

\pscircle(7,1.900){1.25}
\psrotate(7,1.9){45}{\psline{*-*}(7,0.65)(7,3.15)}
\psrotate(7,1.9){-45}{\psline{*-*}(7,0.65)(7,3.15)}
\psarc[linecolor=black]{*-*}(5.6,3.2){1.2}{277}{357}
\psrotate(7,1.9){180}{\psarc[linecolor=black]{*-*}(5.6,3.2){1.2}{277}{357}}
\pscircle[linecolor=red](6.12,1.02){0.12}
\rput(6,0.7){\tiny$1$}
\rput(7.2,0.5){\tiny$2$}
\rput(8.1,1){\tiny$3$}
\rput(6.8,3.3){\tiny$4$}

\rput(11.8,2){
\pstree[treesep=0.5cm,levelsep=0.5cm]{\Tc*{0pt}}{\pstree[treesep=1.2cm,levelsep=0.7cm]{\Tc*{1.2pt}}{\Tc*{1.2pt} \pstree[treesep=1.2cm,levelsep=0.7cm]{\Tc*{1.2pt}}{\Tc*{1.2pt} \pstree[treesep=1.2cm,levelsep=0.7cm]{\Tc*{1.2pt}}{\Tc*{1.2pt} \Tc*{1.2pt}}} }}
}
\pscircle[linecolor=black](11.165,2.8){0.12}
\rput(10.53,1.95){\tiny$2$}
\rput(11.15,1.25){\tiny$1$}
\rput(11.8,0.55){\tiny$3$}
\rput(13.1,0.55){\tiny$4$}

\end{pspicture}
\end{center}

\end{figure} \hspace{0.5cm}

\begin{figure}[H]

 \begin{center}
 \hspace{3cm}
\begin{pspicture}(0,0)(12.8,3.2)

\pscircle(-2,1.900){1.25}
\psrotate(-2,1.9){390}{\psarc[linecolor=black]{*-*}(-2,3.2){1.2}{210}{330}}
\psrotate(-2,1.9){0}{\psline{*-*}(-2,0.65)(-2,3.15)}
\psrotate(-2,1.9){330}{\psarc[linecolor=black]{*-*}(-2,3.4){1.2}{216}{324}}
\psrotate(-2,1.9){140}{\psarc[linecolor=black]{*-*}(-2,4.1){2.1}{236}{304}}

\pscircle[linecolor=red](-2.0,0.66){0.12}
\rput(-2,0.4){\tiny$1$}
\rput(-1.25,0.65){\tiny$2$}
\rput(-0.6,2.1){\tiny$3$}
\rput(-1.25,3.2){\tiny$4$}

\rput(2.2,2){
\pstree[treesep=0.5cm,levelsep=0.5cm]{\Tc*{0pt}}{\pstree[treesep=1.2cm,levelsep=0.7cm]{\Tc*{1.2pt}}{\Tc*{1.2pt} \pstree[treesep=1.2cm,levelsep=0.7cm]{\Tc*{1.2pt}}{\Tc*{1.2pt} \pstree[treesep=1.2cm,levelsep=0.7cm]{\Tc*{1.2pt}}{\Tc*{1.2pt} \Tc*{1.2pt}}} }}
}
\pscircle[linecolor=black](1.565,2.8){0.12}
\rput(0.93,1.95){\tiny$3$}
\rput(1.55,1.25){\tiny$1$}
\rput(2.2,0.55){\tiny$2$}
\rput(3.5,0.55){\tiny$4$}

\psrotate(7,1.9){90}{
\pscircle(7,1.900){1.25}
\psrotate(7,1.9){45}{\psline{*-*}(7,0.65)(7,3.15)}
\psarc[linecolor=black]{*-*}(5.6,3.2){1.5}{277}{357}
\psrotate(7,1.9){180}{\psarc[linecolor=black]{*-*}(5.6,3.2){1.2}{277}{357}}
}
\psarc[linecolor=black]{*-*}(7.25,0.55){0.7}{43}{168}
\pscircle[linecolor=red](6.57,0.69){0.12}
\rput(6.57,0.45){\tiny$1$}
\rput(7.15,0.45){\tiny$2$}
\rput(8.3,2.2){\tiny$3$}
\rput(8,3){\tiny$4$}

\rput(11.8,2){
\pstree[treesep=0.5cm,levelsep=0.5cm]{\Tc*{0pt}}{\pstree[treesep=1.2cm,levelsep=0.7cm]{\Tc*{1.2pt}}{\Tc*{1.2pt} \pstree[treesep=1.2cm,levelsep=0.7cm]{\Tc*{1.2pt}}{\Tc*{1.2pt} \pstree[treesep=1.2cm,levelsep=0.7cm]{\Tc*{1.2pt}}{\Tc*{1.2pt} \Tc*{1.2pt}}} }}
}
\pscircle[linecolor=black](11.165,2.8){0.12}
\rput(10.53,1.95){\tiny$1$}
\rput(11.15,1.25){\tiny$3$}
\rput(11.8,0.55){\tiny$2$}
\rput(13.1,0.55){\tiny$4$}

\end{pspicture}
\end{center}

\end{figure} \hspace{0.5cm}

\begin{figure}[H]

 \begin{center}
 \hspace{3cm}
\begin{pspicture}(0,0)(12.8,3.2)

\psrotate(-2,1.9){90}{
\pscircle(-2,1.900){1.25}
\psrotate(-2,1.9){200}{\psarc[linecolor=black]{*-*}(-2,3.45){1.1}{217}{323}}
\psrotate(-2,1.9){260}{\psarc[linecolor=black]{*-*}(-2,3.45){0.9}{217}{323}}
\psrotate(-2,1.9){320}{\psarc[linecolor=black]{*-*}(-2,3.45){0.9}{217}{323}}
\psrotate(-2,1.9){380}{\psarc[linecolor=black]{*-*}(-2,3.45){0.9}{217}{323}}
}
\pscircle[linecolor=red](-2.845,0.88){0.12}
\rput(-2.845,0.6){\tiny$1$}
\rput(-0.75,1.3){\tiny$2$}
\rput(-1.05,2.85){\tiny$3$}
\rput(-2.5,3.31){\tiny$4$}

\rput(2.2,2){
\pstree[treesep=0.5cm,levelsep=0.5cm]{\Tc*{0pt}}{\pstree[treesep=1.2cm,levelsep=0.7cm]{\Tc*{1.2pt}}{\Tc*{1.2pt} \pstree[treesep=1.2cm,levelsep=0.7cm]{\Tc*{1.2pt}}{\Tc*{1.2pt} \pstree[treesep=1.2cm,levelsep=0.7cm]{\Tc*{1.2pt}}{\Tc*{1.2pt} \Tc*{1.2pt}}} }}
}
\pscircle[linecolor=black](1.565,2.8){0.12}
\rput(0.93,1.95){\tiny$2$}
\rput(1.55,1.25){\tiny$3$}
\rput(2.2,0.55){\tiny$1$}
\rput(3.5,0.55){\tiny$4$}

\psrotate(7,1.9){90}{
\pscircle(7,1.900){1.25}
\psrotate(7,1.9){45}{\psline{*-*}(7,0.65)(7,3.15)}
\psarc[linecolor=black]{*-*}(5.6,3.2){1.5}{277}{357}
\psrotate(7,1.9){180}{\psarc[linecolor=black]{*-*}(5.6,3.2){1.2}{277}{357}}
}
\psrotate(7,1.9){45}{\psline[linecolor=black]{*-*}(7,0.75)(7,3.23)}
\pscircle[linecolor=red](7.06,0.69){0.12}
\rput(7.06,0.45){\tiny$1$}
\rput(8.05,0.95){\tiny$2$}
\rput(8.35,2.2){\tiny$3$}
\rput(8.,2.9){\tiny$4$}

\rput(11.8,2){
\pstree[treesep=0.5cm,levelsep=0.5cm]{\Tc*{0pt}}{\pstree[treesep=1.2cm,levelsep=0.7cm]{\Tc*{1.2pt}}{\Tc*{1.2pt} \pstree[treesep=1.2cm,levelsep=0.7cm]{\Tc*{1.2pt}}{\Tc*{1.2pt} \pstree[treesep=1.2cm,levelsep=0.7cm]{\Tc*{1.2pt}}{\Tc*{1.2pt} \Tc*{1.2pt}}} }}
}
\pscircle[linecolor=black](11.165,2.8){0.12}
\rput(10.53,1.95){\tiny$3$}
\rput(11.15,1.25){\tiny$2$}
\rput(11.8,0.55){\tiny$1$}
\rput(13.1,0.55){\tiny$4$}

\end{pspicture}
\end{center}

\end{figure} \hspace{0.5cm}

\begin{figure}[H]

 \begin{center}
 \hspace{3cm}
\begin{pspicture}(0,0)(12.8,3.2)

\pscircle(-2,1.900){1.25}
\psrotate(-2,1.9){0}{\psline{*-*}(-2,0.65)(-2,3.15)}
\psrotate(-2,1.9){80}{\psline{*-*}(-2,0.65)(-2,3.15)}
\psrotate(-2,1.9){110}{\psline{*-*}(-2,0.65)(-2,3.15)}
\psarc[linecolor=black]{*-*}(-2,0.4){0.9}{33}{146}

\pscircle[linecolor=red](-2.745,0.89){0.12}
\rput(-2.845,0.65){\tiny$1$}
\rput(-2,0.45){\tiny$2$}
\rput(-0.55,1.65){\tiny$3$}
\rput(-0.65,2.3){\tiny$4$}

\rput(2.2,2){
\pstree[treesep=0.5cm,levelsep=0.5cm]{\Tc*{0pt}}{\pstree[treesep=1.4cm,levelsep=0.8cm]{\Tc*{1.2pt}}{\Tc*{1.2pt} \pstree[treesep=1.4cm,levelsep=0.7cm]{\Tc*{1.2pt}}{\pstree[treesep=0.7cm,levelsep=0.5cm]{\Tc*{1.2pt}}{\Tc*{1.2pt} \Tc*{1.2pt}} \pstree[treesep=1.2cm,levelsep=0.6cm]{\Tc*{1.2pt}}{}} }}
}
\pscircle[linecolor=black](1.825,2.8){0.12}
\rput(1.1,1.85){\tiny$1$}
\rput(1.45,0.65){\tiny$2$}
\rput(2.25,0.65){\tiny$3$}
\rput(3.3,1.1){\tiny$4$}

\psrotate(7,1.9){45}{
\pscircle(7,1.900){1.25}
\psrotate(7,1.9){10}{\psline{*-*}(7,0.65)(7,3.15)}
\psarc[linecolor=black]{*-*}(5.6,3.2){1.5}{277}{357}
\psrotate(7,1.9){-60}{\psarc[linecolor=black]{*-*}(5.6,3.2){1.6}{277}{357}}
\psrotate(7,1.9){150}{\psarc[linecolor=black]{*-*}(5.6,3.2){1.6}{277}{357}}
}

\pscircle[linecolor=red](6.15,0.9){0.12}
\rput(6.15,0.65){\tiny$1$}
\rput(7.1,0.47){\tiny$2$}
\rput(8.15,1.15){\tiny$3$}
\rput(7.5,3.15){\tiny$4$}

\rput(11.8,2){
\pstree[treesep=0.5cm,levelsep=0.5cm]{\Tc*{0pt}}{\pstree[treesep=1.4cm,levelsep=0.8cm]{\Tc*{1.2pt}}{\Tc*{1.2pt} \pstree[treesep=1.4cm,levelsep=0.7cm]{\Tc*{1.2pt}}{\pstree[treesep=0.7cm,levelsep=0.5cm]{\Tc*{1.2pt}}{\Tc*{1.2pt} \Tc*{1.2pt}} \pstree[treesep=1.2cm,levelsep=0.6cm]{\Tc*{1.2pt}}{}} }}
}
\pscircle[linecolor=black](11.425,2.8){0.12}
\rput(10.7,1.85){\tiny$2$}
\rput(11.05,0.65){\tiny$1$}
\rput(11.85,0.65){\tiny$3$}
\rput(12.9,1.1){\tiny$4$}

\end{pspicture}
\end{center}

\end{figure} \hspace{0.5cm}

\begin{figure}[H]

 \begin{center}
 \hspace{3cm}
\begin{pspicture}(0,0)(12.8,3.2)

\psrotate(-2,1.9){90}{
\pscircle(-2,1.900){1.25}
\psrotate(-2,1.9){10}{\psline{*-*}(-2,0.65)(-2,3.15)}

\psrotate(-2,1.9){10}{\psarc[linecolor=black]{*-*}(-2,0.4){0.9}{33}{146}}
\psrotate(-2,1.9){0}{\psarc[linecolor=black]{*-*}(-2,3.45){1.4}{221}{320}}
\psrotate(-2,1.9){40}{\psarc[linecolor=black]{*-*}(-2,3.45){1.55}{223}{318}}
}

\pscircle[linecolor=red](-2.756,0.85){0.12}

\rput(-2.756,0.6){\tiny$1$}
\rput(-1.75,0.5){\tiny$2$}
\rput(-0.8,1.34){\tiny$3$}
\rput(-0.75,2.1){\tiny$4$}

\rput(2.2,2){
\pstree[treesep=0.5cm,levelsep=0.5cm]{\Tc*{0pt}}{\pstree[treesep=1.4cm,levelsep=0.8cm]{\Tc*{1.2pt}}{\Tc*{1.2pt} \pstree[treesep=1.4cm,levelsep=0.7cm]{\Tc*{1.2pt}}{\pstree[treesep=0.7cm,levelsep=0.5cm]{\Tc*{1.2pt}}{\Tc*{1.2pt} \Tc*{1.2pt}} \pstree[treesep=1.2cm,levelsep=0.6cm]{\Tc*{1.2pt}}{}} }}
}
\pscircle[linecolor=black](1.825,2.8){0.12}
\rput(1.1,1.85){\tiny$3$}
\rput(1.45,0.65){\tiny$1$}
\rput(2.25,0.65){\tiny$2$}
\rput(3.3,1.1){\tiny$4$}

\pscircle(7,1.900){1.25}
\psrotate(7,1.9){45}{\psline{*-*}(7,0.65)(7,3.15)}
\psrotate(7,1.9){-45}{\psline{*-*}(7,0.65)(7,3.15)}
\psarc[linecolor=black]{*-*}(5.6,3.2){1.2}{277}{357}
\psrotate(7,1.9){180}{\psarc[linecolor=black]{*-*}(5.6,3.2){1.2}{277}{357}}
\pscircle[linecolor=red](7.2,0.65){0.12}
\rput(7.2,0.4){\tiny$1$}
\rput(8.05,0.95){\tiny$2$}
\rput(8.05,2.9){\tiny$3$}
\rput(6.75,3.35){\tiny$4$}

\rput(11.8,2){
\pstree[treesep=0.5cm,levelsep=0.5cm]{\Tc*{0pt}}{\pstree[treesep=1.4cm,levelsep=0.8cm]{\Tc*{1.2pt}}{\Tc*{1.2pt} \pstree[treesep=1.4cm,levelsep=0.7cm]{\Tc*{1.2pt}}{\pstree[treesep=0.7cm,levelsep=0.5cm]{\Tc*{1.2pt}}{\Tc*{1.2pt} \Tc*{1.2pt}} \pstree[treesep=1.2cm,levelsep=0.6cm]{\Tc*{1.2pt}}{}} }}
}
\pscircle[linecolor=black](11.425,2.8){0.12}
\rput(10.7,1.85){\tiny$1$}
\rput(11.05,0.65){\tiny$2$}
\rput(11.85,0.65){\tiny$4$}
\rput(12.9,1.1){\tiny$3$}

\end{pspicture}
\end{center}

\end{figure} \hspace{0.5cm}

\begin{figure}[H]

 \begin{center}
 \hspace{3cm}
\begin{pspicture}(0,0)(12.8,3.2)

\pscircle(-2,1.900){1.25}
\psline{*-*}(-2,0.65)(-2,3.15)
\psarc[linecolor=black]{*-*}(-2,3.15){1}{204}{336}
\psrotate(-2,1.9){180}{\psarc[linecolor=black]{*-*}(-2,3.15){1}{204}{336}}
\psrotate(-2,1.9){240}{\psarc[linecolor=black]{*-*}(-2,3.45){1}{217}{323}}

\pscircle[linecolor=red](-2,0.65){0.12}
\rput(-2,0.4){\tiny$1$}
\rput(-1.55,0.55){\tiny$2$}
\rput(-0.9,1){\tiny$3$}
\rput(-1,2.95){\tiny$4$}

\rput(2.2,2){
\pstree[treesep=0.5cm,levelsep=0.5cm]{\Tc*{0pt}}{\pstree[treesep=1.4cm,levelsep=0.8cm]{\Tc*{1.2pt}}{\Tc*{1.2pt} \pstree[treesep=1.4cm,levelsep=0.7cm]{\Tc*{1.2pt}}{\pstree[treesep=0.7cm,levelsep=0.5cm]{\Tc*{1.2pt}}{\Tc*{1.2pt} \Tc*{1.2pt}} \pstree[treesep=1.2cm,levelsep=0.6cm]{\Tc*{1.2pt}}{}} }}
}
\pscircle[linecolor=black](1.825,2.8){0.12}
\rput(1.1,1.85){\tiny$2$}
\rput(1.45,0.65){\tiny$1$}
\rput(2.25,0.65){\tiny$4$}
\rput(3.3,1.1){\tiny$3$}

\psrotate(7,1.9){225}{
\pscircle(7,1.900){1.25}
\psrotate(7,1.9){10}{\psline{*-*}(7,0.65)(7,3.15)}
\psarc[linecolor=black]{*-*}(5.6,3.2){1.5}{277}{357}
\psrotate(7,1.9){-60}{\psarc[linecolor=black]{*-*}(5.6,3.2){1.6}{277}{357}}
\psrotate(7,1.9){150}{\psarc[linecolor=black]{*-*}(5.6,3.2){1.6}{277}{357}}
}
\pscircle[linecolor=red](6.37,0.78){0.12}
\rput(6.37,0.5){\tiny$1$}
\rput(7.75,0.8){\tiny$2$}
\rput(8.05,1.1){\tiny$3$}
\rput(6.75,3.35){\tiny$4$}

\rput(11.8,1.9){
\pstree[treesep=0.5,levelsep=0.5cm]{\Tc*{0pt}}{\pstree[treesep=1.2cm,levelsep=1.2cm]{\Tc*{1.2pt}}{\pstree[treesep=1.2cm,levelsep=0.7cm]{\Tc*{1.2pt}}{\Tc*{1.2pt} \Tc*{1.2pt}} \pstree[treesep=1.2cm,levelsep=0.7cm]{\Tc*{1.2pt}}{\Tc*{1.2pt} \Tc*{1.2pt}}}
}}
\pscircle[linecolor=black](11.87,2.61){0.12}
\rput(9.95,0.55){\tiny$1$}
\rput(11.25,0.55){\tiny$2$}
\rput(12.5,0.55){\tiny$3$}
\rput(13.8,0.55){\tiny$4$}

\end{pspicture}
\end{center}

\end{figure} \hspace{0.5cm}

\begin{figure}[H]

 \begin{center}
 \hspace{3cm}
\begin{pspicture}(0,0)(12.8,3.2)

\pscircle(-2,1.900){1.25}
\psrotate(-2,1.9){10}{\psarc[linecolor=black]{*-*}(-2,3.45){1.4}{221}{319}}
\psrotate(-2,1.9){100}{\psarc[linecolor=black]{*-*}(-2,3.45){1.4}{221}{319}}
\psrotate(-2,1.9){190}{\psarc[linecolor=black]{*-*}(-2,3.45){1.4}{221}{319}}
\psrotate(-2,1.9){280}{\psarc[linecolor=black]{*-*}(-2,3.45){1.4}{221}{319}}

\pscircle[linecolor=red](-2.93,1.1){0.12}
\rput(-3,0.85){\tiny$1$}
\rput(-2.5,0.55){\tiny$2$}
\rput(-1,0.8){\tiny$3$}
\rput(-1,2.95){\tiny$4$}

\rput(2.2,1.9){
\pstree[treesep=0.5,levelsep=0.5cm]{\Tc*{0pt}}{\pstree[treesep=1.2cm,levelsep=1.2cm]{\Tc*{1.2pt}}{\pstree[treesep=1.2cm,levelsep=0.7cm]{\Tc*{1.2pt}}{\Tc*{1.2pt} \Tc*{1.2pt}} \pstree[treesep=1.2cm,levelsep=0.7cm]{\Tc*{1.2pt}}{\Tc*{1.2pt} \Tc*{1.2pt}}}}
}
\pscircle[linecolor=black](2.2,2.61){0.12}
\rput(0.25,0.55){\tiny$1$}
\rput(1.55,0.55){\tiny$3$}
\rput(2.85,0.55){\tiny$2$}
\rput(4.15,0.55){\tiny$4$}

\psrotate(7,1.9){262}{
\pscircle(7,1.900){1.25}
\psrotate(7,1.9){20}{\psline{*-*}(7,0.65)(7,3.15)}
\psarc[linecolor=black]{*-*}(5.6,3.2){1.5}{277}{357}
\psrotate(7,1.9){-60}{\psarc[linecolor=black]{*-*}(5.6,3.2){1.6}{277}{357}}
\psrotate(7,1.9){150}{\psarc[linecolor=black]{*-*}(5.6,3.2){1.6}{277}{357}}
}
\pscircle[linecolor=red](6.37,0.78){0.12}
\rput(6.37,0.5){\tiny$1$}
\rput(7.15,0.5){\tiny$2$}
\rput(8.25,1.6){\tiny$3$}
\rput(8.3,2.2){\tiny$4$}

\rput(11.8,1.9){
\pstree[treesep=0.5,levelsep=0.5cm]{\Tc*{0pt}}{\pstree[treesep=1.2cm,levelsep=1.2cm]{\Tc*{1.2pt}}{\pstree[treesep=1.2cm,levelsep=0.7cm]{\Tc*{1.2pt}}{\Tc*{1.2pt} \Tc*{1.2pt}} \pstree[treesep=1.2cm,levelsep=0.7cm]{\Tc*{1.2pt}}{\Tc*{1.2pt} \Tc*{1.2pt}}}
}}
\pscircle[linecolor=black](11.87,2.61){0.12}
\rput(9.95,0.55){\tiny$2$}
\rput(11.25,0.55){\tiny$3$}
\rput(12.5,0.55){\tiny$1$}
\rput(13.8,0.55){\tiny$4$}

\end{pspicture}
\end{center}

\end{figure} \hspace{0.5cm}

\begin{figure}[H]

 \begin{center}
 \hspace{3cm}
\begin{pspicture}(0,0)(12.8,3.2)

\pscircle(-2,1.900){1.25}
\psrotate(-2,1.9){45}{\psline{*-*}(-2,0.65)(-2,3.15)}
\psrotate(-2,1.9){-45}{\psline{*-*}(-2,0.65)(-2,3.15)}
\psrotate(-2,1.9){55}{\psarc[linecolor=black]{*-*}(-2,3.75){1}{230}{310}}
\psrotate(-2,1.9){-55}{\psarc[linecolor=black]{*-*}(-2,3.75){1}{230}{310}}

\pscircle[linecolor=red](-2.89,1.01){0.12}
\rput(-2.89,0.75){\tiny$1$}
\rput(-1.0,0.8){\tiny$2$}
\rput(-0.55,2){\tiny$4$}
\rput(-2.57,3.25){\tiny$3$}

\rput(2.2,1.9){
\pstree[treesep=0.5,levelsep=0.5cm]{\Tc*{0pt}}{\pstree[treesep=1.2cm,levelsep=1.2cm]{\Tc*{1.2pt}}{\pstree[treesep=1.2cm,levelsep=0.7cm]{\Tc*{1.2pt}}{\Tc*{1.2pt} \Tc*{1.2pt}} \pstree[treesep=1.2cm,levelsep=0.7cm]{\Tc*{1.2pt}}{\Tc*{1.2pt} \Tc*{1.2pt}}}}
}
\pscircle[linecolor=black](2.2,2.61){0.12}
\rput(0.25,0.55){\tiny$1$}
\rput(1.55,0.55){\tiny$4$}
\rput(2.85,0.55){\tiny$2$}
\rput(4.15,0.55){\tiny$3$}

\psrotate(7,1.9){185}{
\pscircle(7,1.900){1.25}
\psrotate(7,1.9){45}{\psline{*-*}(7,0.65)(7,3.15)}
\psarc[linecolor=black]{*-*}(5.6,3.2){1.5}{277}{357}
\psrotate(7,1.9){180}{\psarc[linecolor=black]{*-*}(5.6,3.2){1.2}{277}{357}}
\psrotate(7,1.9){300}{\psarc[linecolor=black]{*-*}(5.6,3.2){1.2}{277}{357}}
}
\pscircle[linecolor=red](5.98,1.01){0.12}
\rput(5.98,0.75){\tiny$1$}
\rput(6.9,0.5){\tiny$2$}
\rput(8.05,1.1){\tiny$3$}
\rput(6.55,3.35){\tiny$4$}

\rput(11.8,1.9){
\pstree[treesep=0.5,levelsep=0.5cm]{\Tc*{0pt}}{\pstree[treesep=1.2cm,levelsep=1.2cm]{\Tc*{1.2pt}}{\pstree[treesep=1.2cm,levelsep=0.7cm]{\Tc*{1.2pt}}{\Tc*{1.2pt} \Tc*{1.2pt}} \pstree[treesep=1.2cm,levelsep=0.7cm]{\Tc*{1.2pt}}{\Tc*{1.2pt} \Tc*{1.2pt}}}
}}
\pscircle[linecolor=black](11.87,2.61){0.12}
\rput(9.95,0.55){\tiny$2$}
\rput(11.25,0.55){\tiny$4$}
\rput(12.5,0.55){\tiny$1$}
\rput(13.8,0.55){\tiny$3$}

\end{pspicture}
\end{center}

\end{figure} \hspace{0.5cm}

\begin{figure}[H]

 \begin{center}
 \hspace{3cm}
\begin{pspicture}(0,0)(12.8,3.2)

\pscircle(-2,1.900){1.25}
\psrotate(-2,1.9){0}{\psline{*-*}(-2,0.65)(-2,3.15)}
\psrotate(-2,1.9){80}{\psline{*-*}(-2,0.65)(-2,3.15)}
\psrotate(-2,1.9){110}{\psline{*-*}(-2,0.65)(-2,3.15)}
\psarc[linecolor=black]{*-*}(-2,0.4){0.9}{33}{146}

\pscircle[linecolor=red](-2,0.65){0.12}
\rput(-1.15,0.65){\tiny$2$}
\rput(-2,0.4){\tiny$1$}
\rput(-0.55,1.65){\tiny$3$}
\rput(-0.65,2.3){\tiny$4$}

\rput(2.2,1.9){
\pstree[treesep=0.5cm,levelsep=0.5cm]{\Tc*{0pt}}{\pstree[treesep=1.2cm,levelsep=0.7cm]{\Tc*{1.2pt}}{\pstree[treesep=1.2cm,levelsep=0.7cm]{\Tc*{1.2pt}}{\pstree[treesep=1.2cm,levelsep=0.7cm]{\Tc*{1.2pt}}{\Tc*{1.2pt} \Tc*{1.2pt}} \Tc*{1.2pt}} \Tc*{1.2pt}}}
}
\pscircle[linecolor=black](2.84,2.71){0.12}
\rput(0.92,0.45){\tiny$1$}
\rput(2.2,0.45){\tiny$3$}
\rput(2.85,1.15){\tiny$4$}
\rput(3.5,1.85){\tiny$2$}

\psrotate(7,1.9){160}{
\pscircle(7,1.900){1.25}
\psrotate(7,1.9){20}{\psline{*-*}(7,0.65)(7,3.15)}
\psarc[linecolor=black]{*-*}(5.6,3.2){1.5}{277}{357}
\psrotate(7,1.9){-60}{\psarc[linecolor=black]{*-*}(5.6,3.2){1.6}{277}{357}}
\psrotate(7,1.9){150}{\psarc[linecolor=black]{*-*}(5.6,3.2){1.6}{277}{357}}
}
\pscircle[linecolor=red](6.88,0.65){0.12}
\rput(6.88,0.4){\tiny$1$}
\rput(7.45,0.6){\tiny$2$}
\rput(8.25,1.6){\tiny$3$}
\rput(8.2,2.55){\tiny$4$}

\rput(11.8,1.9){
\pstree[treesep=0.5cm,levelsep=0.5cm]{\Tc*{0pt}}{\pstree[treesep=1.2cm,levelsep=0.7cm]{\Tc*{1.2pt}}{\pstree[treesep=1.2cm,levelsep=0.7cm]{\Tc*{1.2pt}}{\pstree[treesep=1.2cm,levelsep=0.7cm]{\Tc*{1.2pt}}{\Tc*{1.2pt} \Tc*{1.2pt}} \Tc*{1.2pt}} \Tc*{1.2pt}}}
}
\pscircle[linecolor=black](12.44,2.71){0.12}
\rput(10.52,0.45){\tiny$1$}
\rput(11.8,0.45){\tiny$4$}
\rput(12.46,1.15){\tiny$2$}
\rput(13.11,1.85){\tiny$3$}

\end{pspicture}
\end{center}

\end{figure} \hspace{0.5cm}

\begin{figure}[H]

 \begin{center}
 \hspace{3cm}
\begin{pspicture}(0,0)(12.8,3.2)

\pscircle(-2,1.900){1.25}
\psrotate(-2,1.9){20}{\psline{*-*}(-2,0.65)(-2,3.15)}
\psrotate(-2,1.9){-20}{\psline{*-*}(-2,0.65)(-2,3.15)}
\psarc[linecolor=black]{*-*}(-2,3.15){1}{204}{336}
\psrotate(-2,1.9){180}{\psarc[linecolor=black]{*-*}(-2,3.15){1}{204}{336}}

\pscircle[linecolor=red](-2.43,0.72){0.12}
\rput(-2.43,0.47){\tiny$1$}
\rput(-1.55,0.49){\tiny$2$}
\rput(-0.85,1.0){\tiny$3$}
\rput(-0.9,2.9){\tiny$4$}

\rput(2.2,1.9){
\pstree[treesep=0.5cm,levelsep=0.5cm]{\Tc*{0pt}}{\pstree[treesep=1.2cm,levelsep=0.7cm]{\Tc*{1.2pt}}{\pstree[treesep=1.2cm,levelsep=0.7cm]{\Tc*{1.2pt}}{\pstree[treesep=1.2cm,levelsep=0.7cm]{\Tc*{1.2pt}}{\Tc*{1.2pt} \Tc*{1.2pt}} \Tc*{1.2pt}} \Tc*{1.2pt}}}
}
\pscircle[linecolor=black](2.84,2.71){0.12}
\rput(0.92,0.45){\tiny$1$}
\rput(2.2,0.45){\tiny$2$}
\rput(2.85,1.15){\tiny$4$}
\rput(3.5,1.85){\tiny$3$}

\psrotate(7,1.9){-45}{
\pscircle(7,1.900){1.25}
\psrotate(7,1.9){45}{\psline{*-*}(7,0.65)(7,3.15)}
\psrotate(7,1.9){-45}{\psline{*-*}(7,0.65)(7,3.15)}
\psarc[linecolor=black]{*-*}(5.6,3.2){1.2}{277}{357}
\psrotate(7,1.9){180}{\psarc[linecolor=black]{*-*}(5.6,3.2){1.2}{277}{357}}
}
\pscircle[linecolor=red](6.88,0.65){0.12}
\rput(6.88,0.4){\tiny$1$}
\rput(7.95,0.85){\tiny$2$}
\rput(8.35,1.85){\tiny$3$}
\rput(7.9,2.98){\tiny$4$}

\rput(11.8,1.9){
\pstree[treesep=0.5cm,levelsep=0.5cm]{\Tc*{0pt}}{\pstree[treesep=1.2cm,levelsep=0.7cm]{\Tc*{1.2pt}}{\pstree[treesep=1.2cm,levelsep=0.7cm]{\Tc*{1.2pt}}{\pstree[treesep=1.2cm,levelsep=0.7cm]{\Tc*{1.2pt}}{\Tc*{1.2pt} \Tc*{1.2pt}} \Tc*{1.2pt}} \Tc*{1.2pt}}}
}
\pscircle[linecolor=black](12.44,2.71){0.12}
\rput(10.52,0.45){\tiny$1$}
\rput(11.8,0.45){\tiny$4$}
\rput(12.46,1.15){\tiny$3$}
\rput(13.11,1.85){\tiny$2$}

\end{pspicture}
\end{center}

\end{figure} \hspace{0.5cm}

\begin{figure}[H]

 \begin{center}
 \hspace{3cm}
\begin{pspicture}(0,0)(12.8,3.2)

\pscircle(-2,1.900){1.25}
\psrotate(-2,1.9){20}{\psline{*-*}(-2,0.65)(-2,3.15)}
\psrotate(-2,1.9){-20}{\psline{*-*}(-2,0.65)(-2,3.15)}
\psrotate(-2,1.9){58}{\psline{*-*}(-2,0.65)(-2,3.15)}
\psrotate(-2,1.9){110}{\psline{*-*}(-2,0.65)(-2,3.15)}

\pscircle[linecolor=red](-2.43,0.72){0.12}
\rput(-2.43,0.47){\tiny$1$}
\rput(-1.55,0.49){\tiny$2$}
\rput(-0.85,1.0){\tiny$3$}
\rput(-0.65,2.45){\tiny$4$}

\rput(2.2,1.9){
\pstree[treesep=0.5cm,levelsep=0.5cm]{\Tc*{0pt}}{\pstree[treesep=1.2cm,levelsep=0.7cm]{\Tc*{1.2pt}}{\pstree[treesep=1.2cm,levelsep=0.7cm]{\Tc*{1.2pt}}{\pstree[treesep=1.2cm,levelsep=0.7cm]{\Tc*{1.2pt}}{\Tc*{1.2pt} \Tc*{1.2pt}} \Tc*{1.2pt}} \Tc*{1.2pt}}}
}
\pscircle[linecolor=black](2.84,2.71){0.12}
\rput(0.92,0.45){\tiny$1$}
\rput(2.2,0.45){\tiny$2$}
\rput(2.85,1.15){\tiny$3$}
\rput(3.5,1.85){\tiny$4$}

\psrotate(7,1.9){50}{
\pscircle(7,1.900){1.25}
\psarc[linecolor=black]{*-*}(5.6,3.2){1.5}{277}{357}
\psrotate(7,1.9){-90}{\psarc[linecolor=black]{*-*}(5.6,3.2){1.6}{277}{357}}
\psrotate(7,1.9){210}{\psarc[linecolor=black]{*-*}(5.6,3.2){1.6}{277}{357}}
\psrotate(7,1.9){-60}{\psarc[linecolor=black]{*-*}(5.6,3.2){1.6}{277}{357}}
}
\pscircle[linecolor=red](6.23,0.85){0.12}
\rput(6.23,0.6){\tiny$1$}
\rput(8.25,1.4){\tiny$2$}
\rput(7.95,2.85){\tiny$3$}
\rput(7.35,3.22){\tiny$4$}

\rput(11.8,1.9){
\pstree[treesep=0.5cm,levelsep=0.5cm]{\Tc*{0pt}}{\pstree[treesep=1.4cm,levelsep=0.7cm]{\Tc*{1.2pt}}{\pstree[treesep=1.2cm,levelsep=0.6cm]{\Tc*{1.2pt}}{\Tc*{1.2pt} \pstree[treesep=0.7cm,levelsep=0.5cm]{\Tc*{1.2pt}}{\Tc*{1.2pt} \Tc*{1.2pt} }} \Tc*{1.2pt}}}
}
\pscircle[linecolor=black](12.12,2.56){0.12}
\rput(10.74,1.1){\tiny$2$}
\rput(11.63,0.6){\tiny$1$}
\rput(12.44,0.6){\tiny$4$}
\rput(12.87,1.69){\tiny$3$}

\end{pspicture}
\end{center}

\end{figure} \hspace{0.5cm}

\begin{figure}[H]

 \begin{center}
 \hspace{3cm}
\begin{pspicture}(0,0)(12.8,3.2)

\psrotate(-2,1.9){35}{
\pscircle(-2,1.900){1.25}
\psrotate(-2,1.9){20}{\psline{*-*}(-2,0.65)(-2,3.15)}
\psrotate(-2,1.9){-20}{\psline{*-*}(-2,0.65)(-2,3.15)}
\psarc[linecolor=black]{*-*}(-2,3.15){1}{204}{336}
\psrotate(-2,1.9){180}{\psarc[linecolor=black]{*-*}(-2,3.15){1}{204}{336}}
}

\pscircle[linecolor=red](-2.4,0.69){0.12}
\rput(-2.4,0.45){\tiny$1$}
\rput(-1.75,0.49){\tiny$2$}
\rput(-0.9,1.05){\tiny$3$}
\rput(-1.75,3.35){\tiny$4$}

\rput(2.2,1.9){
\pstree[treesep=0.5cm,levelsep=0.5cm]{\Tc*{0pt}}{\pstree[treesep=1.4cm,levelsep=0.7cm]{\Tc*{1.2pt}}{\pstree[treesep=1.2cm,levelsep=0.6cm]{\Tc*{1.2pt}}{\Tc*{1.2pt} \pstree[treesep=0.7cm,levelsep=0.5cm]{\Tc*{1.2pt}}{\Tc*{1.2pt} \Tc*{1.2pt} }} \Tc*{1.2pt}}}
}
\pscircle[linecolor=black](2.52,2.56){0.12}
\rput(1.14,1.1){\tiny$1$}
\rput(2.02,0.6){\tiny$2$}
\rput(2.83,0.6){\tiny$3$}
\rput(3.27,1.69){\tiny$4$}

\psrotate(7,1.9){140}{
\pscircle(7,1.900){1.25}
\psrotate(7,1.9){20}{\psline{*-*}(7,0.65)(7,3.15)}
\psarc[linecolor=black]{*-*}(5.6,3.2){1.5}{277}{357}
\psrotate(7,1.9){-60}{\psarc[linecolor=black]{*-*}(5.6,3.2){1.6}{277}{357}}
\psrotate(7,1.9){150}{\psarc[linecolor=black]{*-*}(5.6,3.2){1.6}{277}{357}}
}
\pscircle[linecolor=red](6.01,1.04){0.12}
\rput(5.99,0.76){\tiny$1$}
\rput(6.35,0.55){\tiny$2$}
\rput(6.95,0.45){\tiny$3$}
\rput(8.25,2.22){\tiny$4$}

\rput(11.8,1.9){
\pstree[treesep=0.5cm,levelsep=0.5cm]{\Tc*{0pt}}{\pstree[treesep=1.4cm,levelsep=0.7cm]{\Tc*{1.2pt}}{\pstree[treesep=1.2cm,levelsep=0.6cm]{\Tc*{1.2pt}}{\Tc*{1.2pt} \pstree[treesep=0.7cm,levelsep=0.5cm]{\Tc*{1.2pt}}{\Tc*{1.2pt} \Tc*{1.2pt} }} \Tc*{1.2pt}}}
}
\pscircle[linecolor=black](12.12,2.56){0.12}
\rput(10.74,1.1){\tiny$1$}
\rput(11.63,0.6){\tiny$2$}
\rput(12.44,0.6){\tiny$4$}
\rput(12.87,1.69){\tiny$3$}

\end{pspicture}
\end{center}

\end{figure} \hspace{0.5cm}

\begin{figure}[H]

 \begin{center}
 \hspace{3cm}
\begin{pspicture}(0,0)(12.8,3.2)

\psrotate(-2,1.9){0}{
\pscircle(-2,1.900){1.25}
\psrotate(-2,1.9){10}{\psarc[linecolor=black]{*-*}(-2,3.45){1.4}{221}{319}}
\psrotate(-2,1.9){100}{\psarc[linecolor=black]{*-*}(-2,3.45){1.4}{221}{319}}
\psrotate(-2,1.9){190}{\psarc[linecolor=black]{*-*}(-2,3.45){1.4}{221}{319}}
\psrotate(-2,1.9){280}{\psarc[linecolor=black]{*-*}(-2,3.45){1.4}{221}{319}}
}
\pscircle[linecolor=red](-2.58,0.74){0.12}
\rput(-2.58,0.5){\tiny$1$}
\rput(-1.25,0.79){\tiny$2$}
\rput(-0.85,1.35){\tiny$3$}
\rput(-1.02,2.78){\tiny$4$}

\rput(2.2,1.9){
\pstree[treesep=0.5cm,levelsep=0.5cm]{\Tc*{0pt}}{\pstree[treesep=1.4cm,levelsep=0.7cm]{\Tc*{1.2pt}}{\pstree[treesep=1.2cm,levelsep=0.6cm]{\Tc*{1.2pt}}{\Tc*{1.2pt} \pstree[treesep=0.7cm,levelsep=0.5cm]{\Tc*{1.2pt}}{\Tc*{1.2pt} \Tc*{1.2pt} }} \Tc*{1.2pt}}}
}
\pscircle[linecolor=black](2.52,2.56){0.12}
\rput(1.14,1.1){\tiny$2$}
\rput(2.02,0.6){\tiny$1$}
\rput(2.83,0.6){\tiny$4$}
\rput(3.27,1.69){\tiny$3$}

\psrotate(7,1.9){-30}{
\pscircle(7,1.900){1.25}
\psarc[linecolor=black]{*-*}(7.3,0.5){1}{44}{160}
\psrotate(7,1.9){60}{\psarc[linecolor=black]{*-*}(7.3,0.5){1}{44}{160}}
\psrotate(7,1.9){120}{\psarc[linecolor=black]{*-*}(7.3,0.5){1}{44}{160}}
\psrotate(7,1.9){180}{\psarc[linecolor=black]{*-*}(7.3,0.5){1}{44}{160}}
}
\pscircle[linecolor=red](6.85,0.65){0.12}
\rput(6.85,0.4){\tiny$1$}
\rput(7.45,0.6){\tiny$2$}
\rput(8.15,1.2){\tiny$3$}
\rput(8.15,2.62){\tiny$4$}

\rput(11.8,1.9){
\pstree[treesep=0.5cm,levelsep=0.5cm]{\Tc*{0pt}}{\pstree[treesep=1.4cm,levelsep=0.7cm]{\Tc*{1.2pt}}{\pstree[treesep=1.2cm,levelsep=0.6cm]{\Tc*{1.2pt}}{\Tc*{1.2pt} \pstree[treesep=0.7cm,levelsep=0.5cm]{\Tc*{1.2pt}}{\Tc*{1.2pt} \Tc*{1.2pt} }} \Tc*{1.2pt}}}
}
\pscircle[linecolor=black](12.12,2.56){0.12}
\rput(10.74,1.1){\tiny$1$}
\rput(11.63,0.6){\tiny$3$}
\rput(12.44,0.6){\tiny$4$}
\rput(12.87,1.69){\tiny$2$}

\end{pspicture}
\end{center}

\end{figure} \hspace{0.5cm}

\begin{figure}[H]

 \begin{center}
 \hspace{3cm}
\begin{pspicture}(0,0)(12.8,3.2)

\psrotate(-2,1.9){190}{
\pscircle(-2,1.900){1.25}
\psrotate(-2,1.9){200}{\psarc[linecolor=black]{*-*}(-2,3.45){1.1}{217}{323}}
\psrotate(-2,1.9){260}{\psarc[linecolor=black]{*-*}(-2,3.45){0.9}{217}{323}}
\psrotate(-2,1.9){320}{\psarc[linecolor=black]{*-*}(-2,3.45){0.9}{217}{323}}
\psrotate(-2,1.9){360}{\psarc[linecolor=black]{*-*}(-2,3.45){0.9}{217}{323}}
}
\pscircle[linecolor=red](-2.65,0.77){0.12}
\rput(-2.65,0.5){\tiny$1$}
\rput(-2.0,0.45){\tiny$2$}
\rput(-1.8,3.32){\tiny$3$}
\rput(-3.25,2.78){\tiny$4$}

\rput(2.2,1.9){
\pstree[treesep=0.5cm,levelsep=0.5cm]{\Tc*{0pt}}{\pstree[treesep=1.4cm,levelsep=0.7cm]{\Tc*{1.2pt}}{\pstree[treesep=1.2cm,levelsep=0.6cm]{\Tc*{1.2pt}}{\Tc*{1.2pt} \pstree[treesep=0.7cm,levelsep=0.5cm]{\Tc*{1.2pt}}{\Tc*{1.2pt} \Tc*{1.2pt} }} \Tc*{1.2pt}}}
}
\pscircle[linecolor=black](2.52,2.56){0.12}
\rput(1.14,1.1){\tiny$3$}
\rput(2.02,0.6){\tiny$1$}
\rput(2.83,0.6){\tiny$4$}
\rput(3.27,1.69){\tiny$2$}

\end{pspicture}
\end{center}

\end{figure} \hspace{0.5cm}

\bibliographystyle{plain}
\bibliography{chord}

\begin{thebibliography}{10}

\bibitem{Bal}
W.~Balser.
\newblock {\em From divergent power series to analytic functions}, volume 1582
  of {\em Lectures Notes in Mathematics}.
\newblock Springer--Verlag, Berlin, 1994.

\bibitem{BellonII}
M.~P. Bellon and F.~A. Schaposnik.
\newblock Renormalization group functions for the {W}ess-{Z}umino model: up to
  200 loops through {H}opf algebras.
\newblock {\em Nucl.Phys.A}, 800:517--526, 2008.
\newblock arXiv:0801.0727.

\bibitem{BellonDE}
Marc Bellon.
\newblock Approximate differential equations for renormalization group
  functions in models free of vertex divergencies.
\newblock {\em Nucl.Phys.B}, 826:522--531, 2010.
\newblock arXiv:0907.2296.

\bibitem{BellonDSE}
Marc Bellon.
\newblock An efficient method for the solution of {S}chwinger--{D}yson
  equations for propagators.
\newblock {\em Letters in Mathematical Physics}, 94(1):77--86, 2010.
\newblock arXiv:1005.0196.

\bibitem{BShigher}
Marc~P. Bellon and Fidel~A. Schaposnik.
\newblock Higher loop corrections to a {S}chwinger--{D}yson equation.
\newblock arXiv:1205.0022.

\bibitem{bergk}
Christoph Bergbauer and Dirk Kreimer.
\newblock {H}opf algebras in renormalization theory: {L}ocality and
  {D}yson-{S}chwinger equations from {H}ochschild cohomology.
\newblock {\em IRMA Lect. Math. Theor. Phys.}, 10:133--164, 2006.
\newblock arXiv:hep-th/0506190.

\bibitem{bkerfc}
D.J. Broadhurst and D.~Kreimer.
\newblock Exact solutions of {D}yson-{S}chwinger equations for iterated
  one-loop integrals and propagator-coupling duality.
\newblock {\em Nucl. Phys. B}, 600:403--422, 2001.
\newblock arXiv:hep-th/0012146.

\bibitem{cdm}
S.~Chmutov, S.~Duzhin, and J.~Mostovoy.
\newblock {\em Introduction to Vassiliev Knot Invariants}, chapter 3-4.
\newblock Cambridge University Press, 2012.

\bibitem{ck0}
Alain Connes and Dirk Kreimer.
\newblock Hopf algebras, renormalization and noncommutative geometry.
\newblock {\em Commun. Math. Phys.}, 199:203--242, 1998.
\newblock arXiv:hep-th/9808042.

\bibitem{e-fksurvey}
Kurusch Ebrahimi-Fard and Dirk Kreimer.
\newblock Hopf algebra approach to {F}eynman diagram calculations.
\newblock {\em J. Phys. A}, 38:R285--R406, 2005.
\newblock arXiv:hep-th/0510202.

\bibitem{anatomy}
Dirk Kreimer.
\newblock Anatomy of a gauge theory.
\newblock {\em Annals Phys.}, 321:2757--2781, 2006.
\newblock arXiv:hep-th/0509135v3.

\bibitem{etude}
Dirk Kreimer and Karen Yeats.
\newblock An \'etude in non-linear {D}yson-{S}chwinger equations.
\newblock {\em Nucl. Phys. B Proc. Suppl.}, 160:116--121, 2006.
\newblock arXiv:hep-th/0605096.

\bibitem{NWchord}
A.~Nijenhuis and H.~S. Wilf.
\newblock The enumeration of conected graphs and linked diagrams.
\newblock {\em J. Combin. Theory A}, 27:356--359, 1979.

\bibitem{Schord}
P.~R. Stein.
\newblock On a class of linked diagrams, {I}. {E}numeration.
\newblock {\em J. Combin. Theory A}, 24:357--366, 1978.

\bibitem{vBKUY}
Guillaume van Baalen, Dirk Kreimer, David Uminsky, and Karen Yeats.
\newblock The {QED} beta-function from global solutions to {D}yson-{S}chwinger
  equations.
\newblock {\em Ann. Phys.}, 234(1):205--219, 2008.
\newblock arXiv:0805.0826.

\bibitem{vBKUY2}
Guillaume van Baalen, Dirk Kreimer, David Uminsky, and Karen Yeats.
\newblock The {QCD} beta-function from global solutions to {D}yson-{S}chwinger
  equations.
\newblock {\em Ann. Phys.}, 325(2):300--324, 2010.
\newblock arXiv:0805.0826.

\bibitem{vS}
Walter~D. van Suijlekom.
\newblock Renormalization of gauge fields: A {H}opf algebra approach.
\newblock {\em Commun. Math. Phys.}, 276:773--798, 2007.
\newblock arXiv:hep-th/0610137.

\bibitem{Ymem}
Karen Yeats.
\newblock Rearranging {D}yson-{S}chwinger equations.
\newblock {\em Mem. Amer. Math. Soc.}, 211, 2011.

\bibitem{kythesis}
Karen~Amanda Yeats.
\newblock {\em Growth estimates for {D}yson-{S}chwinger equations}.
\newblock PhD thesis, Boston University, 2008.

\end{thebibliography}

\end{document}